\documentclass[10pt]{article} %change fontsize globally between 10-12pt here%
\usepackage[utf8]{inputenc}%Umlaut-display
\usepackage[top=30mm, left=30mm, right=30mm, bottom=30mm]{geometry}
\usepackage{amsmath,amssymb,amsthm}
\usepackage{dsfont}
\usepackage{color}
\usepackage{lineno,hyperref}
\usepackage{yfonts}
\usepackage{appendix}
\usepackage{babel} %Could lead to better hyphenation, etc.
\newtheorem{theorem}{Theorem}[section]
\newtheorem{lem}[theorem]{Lemma}

\newtheorem{kor}[theorem]{Corollary}
\newtheorem{prop}[theorem]{Proposition}
\newtheorem{rem}[theorem]{Remark}
\theoremstyle{definition}%to have non-italized text in defs to set apart from thms
\newtheorem{dfn}[theorem]{Definition}

\numberwithin{equation}{section}%to have eq-numbers within section instead of absolute (1)...(22)

\DeclareMathOperator*{\supp}{supp}
\DeclareMathOperator*{\divv}{div}

\DeclareMathOperator*{\Id}{Id}

\newcommand*{\Bscr}{\mathcal B}
\newcommand*{\Cscr}{\mathcal C}

\newcommand*{\Fscr}{\mathcal F}

\newcommand*{\Mscr}{\mathcal M}

\newcommand*{\Pscr}{\mathcal P}

\newcommand*{\N}{\mathbb{N}}

\newcommand*{\R}{\mathbb{R}}

\newcommand{\ddt}{\frac{d}{dt}}
\newcommand{\ddtau}{\frac{d}{d\tau}}

\definecolor{orange}{rgb}{1.0, 0.55, 0.0} % further colors added

\newcommand{\footremember}[2]{%
	\footnote{#2}
	\newcounter{#1}
	\setcounter{#1}{\value{footnote}}%
}

\title{Nonlinear Fokker--Planck--Kolmogorov equations as gradient flows on the space of probability measures}
\author{%
	Marco Rehmeier (corresponding author)\footremember{alley}{Faculty of Mathematics, Bielefeld University, 33615 Bielefeld, Germany. E-mail: mrehmeier@math.uni-bielefeld.de} \footnote{Scuola Normale Superiore Pisa, Italy}
	\and Michael Röckner\footnote{Faculty of Mathematics, Bielefeld University, 33615 Bielefeld, Germany. E-mail: roeckner@math.uni-bielefeld.de} \footnote{Academy of Mathematics and System Sciences, CAS, Beijing}%
}
\date{}
\begin{document}
	\maketitle
	%\author{Marco Rehmeier\thanks{Faculty of Mathematics, Bielefeld University, 33615 Bielefeld, Germany. E-Mail: mrehmeier@math.uni-bielefeld.de }, Michael Röckner\thanks{Faculty of Mathematics, Bielefeld University, 33615 Bielefeld, Germany. E-Mail: roeckner@math.uni-bielefeld.de }}

	\begin{abstract}
	We propose a general method to identify nonlinear Fokker--Planck--Kolmogorov equations (FPK equations) as gradient flows on the space of probability measures on $\R^d$ with a natural differential geometry. Our notion of gradient flow does not depend on any underlying metric structure such as the Wasserstein distance, but is derived from purely differential geometric considerations. We explicitly identify the associated entropy functions, and also the corresponding energy functionals, in particular their domains of definition. Furthermore, the latter functions are Lyapunov functions for the solutions of the FPK equations. Moreover, we show uniqueness results for such gradient flows, and we also prove that the gradient of $E$ is a gradient field on $\R^d$, which can be approximated by smooth gradient fields. These results cover classical and generalized porous media equations, where the latter have a generalized diffusivity function and a nonlinear transport-type first-order perturbation.
	\end{abstract}
	\noindent	\textbf{Keywords:} Gradient flow, nonlinear Fokker--Planck equation, generalized porous media equation, differential geometry, Barenblatt solution, direct integral of Hilbert spaces.\\
	\textbf{2020 MSC:} 35Q84, 35K55, 76S05, 58B20, 37B35, 35B40. 
	
	\paragraph{Data availability statement:} The authors declare that no relevant data has been created or used for this work. 
	\section{Introduction}
	In this paper we propose a simple general approach to identify solutions to \emph{nonlinear} Fokker--Planck--Kolmogorov equations (FPK equations) of type
	\begin{equation}\label{intro:NLFPKE}\tag{FPK}
		\partial_t \mu_t = \partial_{ij}(a_{ij}(t,\mu_t,x)\mu_t)-\partial_i(b_i(t,\mu_t,x)\mu_t)
	\end{equation}
(where $a_{ij},b_i: \R_+\times \Pscr\times\R^d\to \R$), with measure valued solution paths,
	as solutions to gradient flows on the space $\Pscr$ of Borel probability measures on $\R^d$, $d \geq 1$, i.e. as solutions to equations of type
	\begin{equation}\label{intro:GF}\tag{GF}
		\ddt \mu_t = -\nabla^\Pscr E_{\mu_t}.
	\end{equation} Furthermore, we explicitly identify the associated energy functions $E$, in particular their domains of definition, as Lyapunov functions for these solutions, which can, e.g., be used to prove existence of stationary solutions to the FPK equation (see \cite{BR-IndianaJ}). Our approach does not involve any metric on $\Pscr$, but only a general construction of a natural tangent bundle on $\Pscr$ and corresponding gradient $\nabla^\Pscr$ (as was done in \cite{AKR1,AKR3,Roeckner1998} for configuration spaces replacing $\Pscr$). Therefore, our method is completely different from the well-known approach developed in the theory of optimal transport, e.g. via the theory of gradient flows in metric spaces applied to the Wasserstein space $\Pscr_2$ of Borel probability measures on $\R^d$ with finite second moment, in which case the induced metric, the Wasserstein metric, plays a major role.
The latter approach goes back to the pioneering work by Otto \cite{Otto01} and Jordan, Kinderlehrer, Otto \cite{JKO98}. For very nice presentations of this and related material, see e.g. the books of Ambrosio, Gigli, Savaré \cite{AGS05} and Villani \cite{Villani08-book}, as well as the lecture notes by Figalli and Glaudo \cite{FG21}. 
\\

Although our method is expected to be applicable to more general classes of equations, in this work we focus on \eqref{intro:NLFPKE} with so-called Nemytskii-type coefficients, which depend on the density $u$ (with respect to Lebesgue measure $dx$) of $\mu(dx) = u(x)dx \in \Pscr$ pointwise via $a_{ij}(\mu,x) = \tilde{a}_{ij}(u(x),x)$ and $b_i(\mu,x) = \tilde{b}_i(u(x),x)$ for measurable $\tilde{a}_{ij},\tilde{b}_i:\R\times \R^d \to \R$. More precisely, our main results are on generalized porous media equations (PME) of type
\begin{equation}\label{intro:genPME}\tag{gPME}
	\partial_t u = \Delta \beta (u) - \divv(Db(u)u),
\end{equation}
which are nonlinear FPK equations reformulated as PDEs for the densities $u(t)$ of solutions $t \mapsto \mu_t$. Here, \textit{generalized} PME does not only refer to the generalized diffusivity $\beta: \R \to \R$ (where $\beta(r) = |r|^{m-1}r$ gives the classical PME), but also to the additional nonlinear transport-type first order perturbation, composed of a vector field $D: \R^d\to \R^d$ and the nonlinearity $b: \R \to \R$. The papers \cite{NLFPK-DDSDE5}-\cite{BR23} further developed the existence- and uniqueness theory from the pioneering results in \cite{BBC75,BC79,Pierre82,BF83} on these type of equations, and, in particular, were devoted to the relation with McKean--Vlasov type stochastic differential equations. Though solutions to these equations have densities with respect to $dx$ for $t >0$, under suitable assumptions on $\beta$ their initial data may be arbitrary probability measures. To the best of our knowledge, the present paper is the first in which equations of this general type are identified as gradient flows on a space of measures.
\\

The differential geometry on $\Pscr$ used in this paper was first introduced in \cite{RRW20}. But the method how to identify natural tangent bundles over "manifold-like" state spaces $\Mscr$ has been known much longer and goes back to \cite{AKR1,AKR3,Roeckner1998}. It goes in two steps. The first is to fix a large enough (i.e., at least point separating) class $\Fscr$ of "test functions" $F:\mathcal{M}\to \R$. The second is to fix for each $x \in \mathcal{M}$ a set $\Cscr$ of "suitable" curves $\gamma^x:(-\varepsilon,\varepsilon)\to \mathcal{M}$ such that $\gamma^x(0)=x$, along which one can differentiate $t\mapsto F(\gamma^x(t))$ at $t=0$ for each $F\in \Fscr$, which gives derivations at $x$ as linear maps on $\Fscr$, i.e. a "reduced tangent space" at $x$. Here "reduced" indicates the fact that if $\mathcal{M}$ is a Riemannian manifold, then depending on $\Fscr \subseteq C^1(\Mscr)$ and $\Cscr$, one obtains smaller tangent spaces than the usual $T_x\Mscr$. Finding $\Fscr$ and, in particular, $\Cscr$, in the case that $\mathcal{M}$ is a space of measures on a Riemannian manifold $M$, is possible due to the general idea of "lifting" the geometry from $M$ to $\mathcal{M}$, which goes as follows: For a smooth, compactly supported section $\varphi$ in the tangent bundle $TM$ of $M$ (we write $\varphi \in C^\infty_c(M,TM)$), let $\Phi^\varphi$ be the flow of $\varphi$ on $M$, i.e. 
$$\Phi^\varphi(0,x) = x,\quad \ddt \Phi^\varphi(t,x) = \varphi(\Phi^\varphi(t,x)),\quad \forall (t,x) \in \R\times M.$$
 For $\nu \in \mathcal{M}$, let $\tilde{\mu}^{\varphi,\nu}_t := \nu\circ \Phi^\varphi (t)^{-1}$. The curves $t\mapsto \tilde{\mu}^{\varphi,\nu}_t$, $\varphi \in C_c^\infty(M,TM)$, form such a set $\Cscr$ of suitable curves. Then let $\Fscr$ be all functions of the form $F: \mathcal{M}\to \R$, $F(\nu)=f(\nu(h_1),\dots,\nu(h_k))$, $h_i \in C^\infty_c(M,\R)$, $f \in C^1_b(\R^k)$, $k \in \N$. One sets 
 $$\nabla_\varphi F(\nu):= \ddt F(\tilde{\mu}^{\varphi,\nu}_t)_{|t=0}$$
  and obtains by the chain rule
\begin{equation*}
	\nabla_\varphi F(\nu) = \int_M\bigg \langle \sum_{i=1}^k\partial_if(\nu(h_1),\dots,\nu(h_k))\nabla^M h_i, \varphi\bigg\rangle_{T_xM} \, d\nu(x).
\end{equation*}
Hence one defines 
$$T_\nu \mathcal{M} := \overline{C^\infty_c(M,TM)}^{L^2(M,TM;\nu)} = L^2(M,TM);\nu),\,\,\nu \in \Mscr,$$
 with inner product 
 $$\langle \varphi, \bar{\varphi} \rangle_{T_\nu \mathcal{M}}:= \int_M \langle \varphi(x),\bar{\varphi}(x)\rangle_{T_xM}\,d\nu(x).$$
 Then the gradient $\nabla^\mathcal{M}F(\nu) \in T_\nu\Mscr$ at $\nu \in \Mscr$ is uniquely characterized by
\begin{equation}
	\nabla_\varphi F(\nu) = \langle \nabla^\mathcal{M}F(\nu), \varphi\rangle_{T_\nu \mathcal{M}},\quad \forall \varphi \in T_\nu \Mscr,
\end{equation}
i.e. for $F$ as above $\nabla^\mathcal{M}F(\nu) = \sum_{i=1}^k\partial_if(\nu(h_1),\dots,\nu(h_k)) \nabla^M h_i$. The first implementation of this scheme was done in \cite{AKR1,AKR3,Roeckner1998} to deduce a natural geometry on the configuration space $\Gamma(M)$ of a Riemannian manifold $M$, i.e. on the space of $\mathbb{Z}_+\cup \{\infty\}$-valued Radon measures on $M$.

For the case $M = \R^d$, $\mathcal{M} = \Pscr$, which we consider in the present paper following \cite{RRW20}, we may replace $\tilde{\mu}^{\varphi,\nu}_t$ by 
$$\mu^{\varphi,\nu}_t := \nu\circ (\Id+t\varphi)^{-1},\quad \forall \varphi \in C^\infty_c(\R^d,\R^d), \,\nu \in \Pscr $$
since clearly ${\ddt \tilde{\mu}^{\varphi,\nu}_t}_{|t=0}  = {\ddt \mu^{\varphi,\nu}_t}_{|t=0}$. The definition of $\mu^{\varphi,\nu}$ can be extended to every $\varphi\in L^2(\R^d,\R^d;\nu)$, whereby we arrive at our tangent spaces $T_\nu \Pscr = L^2(\R^d,\R^d;\nu)$, $\nu\in \Pscr$. Depending on the FPK equation \eqref{intro:NLFPKE} under consideration, we consider $T_\nu\Pscr$ with either the standard $L^2$-inner product $\langle \cdot, \cdot \rangle_\nu$ or, unlike \cite{RRW20}, with a weighted inner product with weight at $\nu(dx)= v(x)dx$ given by $\frac{1}{b(v)}$, see \eqref{def:weighted-tensor1}-\eqref{def:weighted-tensor2} below. We refer to Section \ref{subsect:diff-geom} and the appendix for more details on our geometry. That the rather small class $\{t\mapsto \mu^{\varphi,\nu}_t, \varphi \in L^2(\R^d,\R^d;\nu), \nu \in \Pscr\}$ of curves is sufficiently large and indeed the "right" class to give a gradient flow representation for all nonlinear FPK equations of type \eqref{intro:genPME} is rather surprising, but this is the core result of the present paper.
\\

Our first main aim in this paper is to present a general technique that reveals the gradient flow structure of nonlinear FPK equations with respect to the aforementioned differential geometry and identifies the corresponding entropy functions, and also the corresponding energy functionals. The key idea is to evaluate solutions to nonlinear FPK equations, which are weakly continuous paths $t \mapsto \mu_t$ in $\Pscr$, not only through the linear functions $t \mapsto \mu_t(\zeta) := \int_{\R^d} \zeta \,d\mu_t$, $\zeta \in C^2_c(\R^d)$, but through the much larger class of nonlinear finitely based functions on $\Pscr$, i.e. through all $F:\Pscr \to \R$ of type $F(\mu)= f(\mu(h_1),\dots,\mu(h_k)), k \in \N, h_i \in C^2_c(\R^d), f \in C^1_b(\R^k)$ (see the Notation section at the end of this introduction for the notation of the usual function spaces used here and throughout). More precisely, one simply calculates $\frac d {dt} F(\mu_t)$ using the fact that $t \mapsto \mu_t$ solves the nonlinear FPK equation.To obtain that $t\mapsto \mu_t$ satisfies \eqref{intro:GF}, this derivative should equal $-\text{diff}E_{\mu_t}(\nabla^\Pscr F_{\mu_t})$ for all $F$ as above, which provides an ansatz to find $E$ and hence the right gradient flow equation (see Section \ref{subsect:lifting-reveals} below for details).

The second aim of this paper is to implement this technique rigorously for the class of FPK equations in  \eqref{intro:genPME}. We would like to stress that this, in particular, requires to prove that the domain of $\nu \mapsto \text{diff}E_\nu$ is large enough for the energy functions $E$ corresponding to \eqref{intro:genPME}. This constitutes a core of our result, turning our technique into a rigorous mathematical proof rather than a merely purely heuristic computation. Our method seems also applicable for other classes of FPK equations. Since the corresponding energy is a Lyapunov function for the solutions, which can be used to analyze the asymptotic behavior of the gradient flow solutions (see \cite{BR-IndianaJ}), we hope that one benefit of this paper will be to identify Lyapunov functions for more general classes of nonlinear FPK equations and, thereby, to obtain new results for the asymptotics of their solutions.
The following first main theorem of this paper concerns the classical PME with arbitrary (not necessarily absolutely continuous with respect to Lebesgue measure) probability measures as initial data.
\paragraph{Theorem 1. (see Theorem \ref{thm:classPME-case} below for the precise formulation)}
Let $\mu_0 \in \Pscr$ and $m \geq 2$. The unique probability solution (see Definition \ref{def:distr-sol-gen-NLFPKE} below) to the classical PME
\begin{equation}\label{intro:classPME}\tag{PME}
	\partial_t u = \Delta(|u|^{m-1}u),\quad t >0
\end{equation}
 in $\bigcap_{\delta >0}L^\infty((\delta,\infty)\times \R^d)$ is the restricted unique solution to \eqref{intro:GF} on $\Pscr$ with energy $E_\nu= \frac{1}{m-1}\int_{\R^d}v(x)^{m} dx$ and gradient $\nabla^\Pscr E_\nu = \frac{\nabla (v^{m})}{v}$ ($=\frac{m}{m-1} \nabla(v^{m-1}$), if $\nu(dx) = v(x)dx$ with $v^{m-1} \in W^{1,1}_{\text{loc}}(\R^d)$, which is the case for the unique solution to \eqref{intro:classPME} for $m\geq 3$).
\\
\\
 In particular, the famous Barenblatt solutions turn out to satisfy a gradient flow equation on $\Pscr$. Our energy $E$ coincides with the energy function in \cite{Otto01}, 
 where a very nice physical interpretation of $E$ is presented, though only absolutely continuous initial measures have been considered there.
As the second main theorem, we have the following extension to the nonlinear FPK equations of type \eqref{intro:genPME} as follows.
\paragraph{Theorem 2. (see Theorems \ref{thm1} and \ref{prop:main-appl-general-case} below for the precise formulation)}
Under suitable assumptions on $\beta, D=-\nabla \Phi$ and $b$, the unique probability solution to \eqref{intro:genPME} is the restricted unique solution to \eqref{intro:GF} in $L^\infty([0,\infty)\times \R^d)$ with energy 
$$
E_\nu:= \int_{\R^d} \eta(v(x))\,dx + \int_{\R^d} \Phi(x)v(x)dx,\quad \nu(dx) = v(x)dx \in D_0,$$
where
$$ \eta(r):= \int_0^r g(s)ds := \int_0^r \int_1^s\frac{\beta'(w)}{wb(w)}dw \, ds,
$$
and gradient $\nabla^\Pscr_b E_\nu = b(v)\nabla \big(g(v)+\Phi\big)$ with respect to a weighted metric tensor $\langle \cdot, \cdot \rangle_b$ depending on $b$.
Moreover, $E$ is a Lyapunov function for $u$, i.e. $E(u(t))\leq E(u(s))$ for all $s \leq t$.
 \\
 \\
 For the definition of thte domain $D_0$, we refer to Theorem \ref{thm1}.
 As mentioned before, due to the nonlinear transport-type perturbation, we are led to consider the weighted $L^2$-metric tensor $\langle \cdot, \cdot \rangle_b$ with weight at $\nu(dx) = v(x)dx$ given by $\frac 1 {b(v)}$. 
 Furthermore, we prove for $\nu \in D_0$ that $\nabla^\Pscr_b E_\nu$ is a gradient vector field on $\R^d$ which, in addition, belongs to the closure of $\{b(v)\nabla \zeta\,|\, \zeta \in C^\infty_c(\R^d)\}$ in $L^2(\R^d,\R^d;\nu)$ (see Proposition \ref{prop:Michael-grad} below).
 We note that in the classical case of the heat equation, i.e. $\beta(r)=r$ and $\Phi=0$, $E$ is the classical Boltzmann entropy function.
 \\
 
We would like to repeat that our approach is substantially different from \cite{Otto01,JKO98} and subsequent related work, such as \cite{AGS05,Villani08-book,FG21,E10} and \cite{GH23}. In the latter, based on \cite{DSZ16}, by a large deviation principle the authors consider the classical PME as a gradient flow in a modified (with respect to \cite{Otto01}) Wasserstein-type geometry. We would also like to mention analogous results for discrete equations, for instance \cite{M11,EM14} and follow-up papers, which, however, are also quite different from the present work. 
For a different approach to gradient flows without relying on the Wasserstein metric, we refer to the substantial work \cite{PRST22}.
 A more detailed overview of the available literature on these directions is beyond the scope of this introduction. The point we want to stress is that all these works are indeed of a different flavor than the present paper, in which we identify FPK equations as gradient flows with respect to the differential geometry obtained by the previously mentioned lifting procedure, but without involving any metric.

\paragraph{Organization.}

The rest of this paper is organized as follows. We repeat the general notion of gradient flows on Riemannian manifolds in Section \ref{subsect:gradflow-general} and introduce our geometry on $\Pscr$ in \ref{subsect:diff-geom}. In \ref{subsect:lifting-reveals} we explain our ansatz to reveal the gradient flow structure of nonlinear FPKEs, and we present the notion of gradient flow in our differential geometry on $\Pscr$ in \ref{subsect:gradflow-P}. Section \ref{sect:results} contains the main results. First, we present our result on the classical PME and equations with general diffusivity functions in \ref{subsect:classPME}, then those on generalized PMEs in \ref{sect:main-sect:general-case} including nonlinear transport-type drifts, and, finally, in Section \ref{sect:matrix-case} we briefly discuss an extension to more general divergence-type equations. In the appendix we present a further natural deduction of our differential geometry.
	
	\paragraph{Notation.}
	We write $\Pscr := \Pscr(\R^d)$ for the set of Borel probability measures on $\R^d$ and $\Pscr_a \subseteq \Pscr$ for its subset of absolutely continuous measures with respect to Lebesgue measure $dx$.
	For a measure $\nu$ on a measurable space $(X,\mathcal{X})$ and an $\mathcal{X}$-measurable function $f: X \to \R$, we abbreviate $\nu(f):= \int_X f(x)\,d\nu(x)$, provided the integral is defined. A \textit{Borel curve} in $\Pscr$ is a curve $t\mapsto \mu_t$ from an interval $I \subseteq \R$ such that $t\mapsto \mu_t(A)$ is measurable for all Borel sets $A\in \Bscr(\R^d)$. Recall that the \textit{topology of weak convergence of measures} is the initial topology of the maps $\mu \mapsto \mu(g)$ for all bounded continuous $g: \R^d\to \R$. Restricted to $\Pscr$, it suffices to consider smooth compactly supported $g$.
	
	We use the following standard function space notation. $C^m(\R^d,\R^k)$, $C^m_b(\R^d,\R^k)$ and $C^m_c(\R^d,\R^k)$, $m \in \N_0 \cup \{\infty\}$, denote the spaces of $m$-fold differentiable functions $g: \R^d\to \R^k$ and, respectively, their subsets of bounded and compactly supported functions. For $k=1$, we write $C^m(\R^d),C^m_b(\R^d)$ and $C^m_c(\R^d)$, and for $m=0$ $C(\R^d,\R^k)$, $C_b(\R^d,\R^k)$ and $C_c(\R^d,\R^k)$, respectively. For $p \in [1,\infty]$, an open set $U \subseteq \R^d$ and a measure $\nu$ on $\Bscr(\R^d)$, the usual $L^p$-spaces of (equivalence classes of) Borel measurable functions $g: U\to \R^k$ with (locally) $\nu$-integrable $p$-th power are denoted by $L_{(\text{loc})}^p(U,\R^k;\nu)$. If either $\nu = dx$ or $k=1$, we simply write $L_{(\text{loc})}^p(U,\R^k)$ or $L^p_{(\text{loc})}(U;\nu)$, respectively. For the usual associated norms, we write $|\cdot|_p$, if no confusion about $d$, $k$ or $\nu$ can occur. Moreover, we denote by $W^{m,p}_{(\text{loc})}(\R^d)$ the usual Sobolev spaces of functions $g: \R^d \to \R$ with weak partial derivatives up to order $m$ in $L^p_{(\text{loc})}(\R^d)$. For $p=2$, we write $H_{(\text{loc})}^m(\R^d)$. These Sobolev spaces are always considered with respect to $dx$. $\Bscr_b(\R^d)$ is the space of Borel measurable, bounded functions $g: \R^d\to \R$.
		
	The usual Euclidean norm and inner product on $\R^d$ are denoted by $|\cdot|$ and $x\cdot y$. Depending on context, we write $\Id$ both for the identity vector field $\Phi: \R^d \to \R^d$, $\Phi(x) = x$, and for the $d\times d$-identity matrix. For $f: \R^d \to \R$, we write $f^+:= \max(0,f)$ and $f^-:= \max(0,-f)$. We set $\R_+ := [0,\infty)$ and, for $x,y \in \R$, $x\wedge y := \min(x,y)$ and $x\vee y := \max(x,y)$.
	
\section{Gradient flows on Riemannian manifolds and on $\Pscr$}\label{sect:2}
Here we briefly recall the notion of gradient flows on Riemannian manifolds from a purely differential geometric perspective. Then we introduce our geometry on $\Pscr$ and the notion of gradient flow on $\Pscr$ based on it.
\subsection{Gradient flows on Riemannian manifolds}\label{subsect:gradflow-general}
Let $M$ be a $d$-dimensional Riemannian manifold, $C^1(M)$ the space of differentiable functions $F: M \to \R$, $T_xM$ the (Hilbert) tangent space at $x \in M$ (i.e. $\ell \in T_xM$ if and only if $\ell: C^1(M)\to \R$ such that $\ell(FG)= \ell(F)G(x)+\ell(G)F(x)$ for all $F,G \in C^1(M)$) with inner product $\langle \cdot, \cdot \rangle_{T_xM}$ and dual space $T_xM^*$, and let $\text{diff}F_x$ be the differential of $F$ at $x$. Elements $\ell \in T_xM$ are called \textit{derivations} and act on $F \in C^1(M)$ via
\begin{align*}
	\ell(F) = \ddtau (F \circ \gamma^\ell(\tau))_{\tau=0} = \text{diff}F_x\bigg(\ddtau \gamma^\ell(\tau)_{|\tau=0}\bigg) = \big\langle \nabla F_x , \ddtau \gamma^\ell(\tau)_{|\tau=0}\big\rangle_{T_xM} = \big \langle \nabla F_x , \ell \big\rangle_{T_xM}.
\end{align*}
Here $ \gamma^\ell$ denotes any $C^1$-curve $\gamma^\ell: (-\varepsilon,\varepsilon)\to M$  such that the first equality holds for all $F \in C^1(M)$. At least one such curve with $\gamma^\ell (0) = x$ exists, since $T_xM$ can equivalently be defined as the set of equivalence classes of $C^1$-curves $\gamma: (-\varepsilon,\varepsilon)\to M$ with $\gamma(0) = x$, with equivalence relation $\gamma_1 \sim \gamma_2 :\iff \ddtau (F \circ \gamma_1(\tau))_{|\tau=0}  = \ddtau (F \circ \gamma_2(\tau))_{|\tau=0}$ for all $F \in C^1(M)$, i.e. there is an isomorphism between the set of derivations $\ell$ and such equivalence classes. The second equality is the definition of the differential $\text{diff}F$. Since $\text{diff}F_x\in T_xM^*$, the third equality follows from Riesz' representation theorem and uniquely characterizes the gradient $\nabla F_x \in T_xM$. The final equality follows from the aforementioned equivalence of both definitions of $T_xM$. In particular, letting $\ell  = -\nabla E_x$ for $E \in C^1(M)$, it follows that the derivation $-\nabla E_x$ acts on $F\in C^1(M)$ via
\begin{equation*}
	-\nabla E_x(F) = -\text{diff}E_x(\nabla F_x)= -\langle \nabla F_x, \nabla E_x\rangle_{T_xM}.
\end{equation*}
Let $I = (a,b)$ with $-\infty \leq a < b \leq +\infty$. A \textit{gradient flow} on $M$ is an equation
\begin{equation}\label{gradflow-general}
	\frac{d}{d\tau} x(\tau)_{|\tau=t} = -\nabla E_{x(t)},\quad\forall  t \in I
\end{equation}
in the tangent bundle $TM=\bigsqcup_{x \in M}T_xM$, to be solved for differentiable curves $t \mapsto x(t)$ on $M$. More precisely, the differentiability of $t\mapsto x(t)$ yields $\ddtau x(\tau)_{|\tau =t} \in T_{x(t)}M$ and, as explained above, $-\nabla E_{x(t)}$ acts via
\begin{equation*}
	-\nabla E_{x(t)}(F) = -\langle \nabla E_{x(t)},\nabla F_{x(t)}\rangle_{T_{x(t)}M},\quad \forall F \in C^{1}(M).
\end{equation*}Hence \eqref{gradflow-general} implies
\begin{equation}\label{GF-fct}
	\ddtau F\big(x(\tau)\big)_{|\tau =t} = -\text{diff}E_x(\nabla F_x)=-\langle \nabla E_{x(t)},\nabla F_{x(t)}\rangle_{T_{x(t)}M},\quad\forall t\in I, F \in C^1(M),
\end{equation}
which is, in fact, equivalent to \eqref{gradflow-general}. Indeed, for $\xi, \xi' \in T_xM$, $\text{diff}F_x(\xi) = \text{diff}F_x(\xi')$ for all $F\in C^1(M)$ implies $\xi = \xi'$, since every cotangent vector at $x$ is the differential at $x$ of a function in $C^1(M)$, and the set of cotangent elements at $x$ separates tangent vectors at $x$.
$E$ in \eqref{gradflow-general} is also called \textit{energy (function)} of the system, and choosing $F = E$ in \eqref{GF-fct} shows
$$E(x(t))-E(x(s)) = -\int_s^t |\nabla E_{x(r)}|^2_{T_{x(r)}M} \,dr,\quad \forall s \leq t \in I,$$
i.e. for any solution $x$ of \eqref{gradflow-general}, $t\mapsto E(x(t))$ is non-increasing.

\subsection{Differential geometry on $\Pscr$}\label{subsect:diff-geom}
As mentioned in the introduction, we consider $\Pscr$ as a manifold-like space with the lifted geometry from $\R^d$. Originally, an analogous geometry was introduced in \cite{AKR1,AKR3,Roeckner1998} on the space $\Gamma(M)$ of $\mathbb{Z}_+\cup \{+\infty\}$-valued Radon measures on a Riemannian manifold $M$ (see also the follow-up papers \cite{ORS95,AKR2,AKR4,OR97,R98,MaRoeck01,ADKL03}). The same approach for $M=\R^d$ and $\Gamma(M)$ replaced by $\Pscr = \Pscr(\R^d)$, which was first implemented in \cite[App.A]{RRW20} and was recalled in the introduction of this paper, leads to the geometry of the present paper. Here we summarize the resulting geometry without repeating in detail the general idea of its construction.
The test function class consists of the finitely-based functions
\begin{equation}\label{test-fcts}
	\Fscr C^2_b := \big\{F:\Pscr \to \R: F(\nu)= f(\nu(h_1),\dots,\nu(h_k)), k \in \N, h_i \in C^2_c(\R^d), f\in C^1_b(\R^k)\big\},
\end{equation}
and the class of differentiable curves $(-\varepsilon,\varepsilon)\ni t \mapsto \mu_t$ on $\Pscr$ passing through $\nu \in \Pscr$ at $t=0$ is given by all curves of type 
$$\mu^{\varphi,\nu}_t := \nu \circ (\Id+t\varphi)^{-1},\quad \varphi \in L^2(\R^d,\R^d;\nu).$$
Consequently, the tangent bundle $T\Pscr := \bigsqcup_{\nu \in \Pscr}T_\nu\Pscr$ consists of the Hilbert tangent spaces
\begin{equation}\label{tangent-spaces-P}
	T_\nu\Pscr := L^2(\R^d,\R^d;\nu),
\end{equation}
with metric tensor
\begin{equation}\label{metric-tensor-P}
	\langle \cdot, \cdot \rangle : \nu \mapsto \langle \cdot, \cdot \rangle_\nu
\end{equation}
on $T\Pscr$,
where $\langle \cdot, \cdot \rangle_\nu$ denotes the standard inner product on $L^2(\R^d,\R^d;\nu)$. Indeed, $\varphi \in T_\nu\Pscr$ acts as a derivation on $F\in \Fscr C^2_b, F(\mu) = f(\mu(h_1),\dots,\mu(h_k))$, because by the chain rule
\begin{equation*}\label{action-of-tangent-element on F}
	\ddt F(\mu^{\varphi,\nu}_t)_{|t=0} = \sum_{i=1}^k \partial_i f(\nu(h_1),\dots,\nu(h_k)) \langle \nabla h_i, \varphi \rangle_\nu.
\end{equation*}
The \textit{differential} of $F \in \Fscr C^2_b$ at $\nu \in \Pscr$ is the continuous linear functional on $T_\nu \Pscr$
\begin{equation}\label{diff-P}
	\text{diff} F_\nu: \varphi \mapsto \ddt F(\mu^{\varphi,\nu}_t)_{|t=0} = \sum_{i=1}^k \partial_i f(\nu(h_1),\dots,\nu(h_k)) \langle \nabla h_i, \varphi \rangle_\nu,
\end{equation}
and, in analogy to Riemannian geometry, the \textit{gradient} $\nabla ^\Pscr F$ at $\nu$ is defined as the unique element $\nabla ^\Pscr F_{\nu} \in L^2(\R^d,\R^d;\nu)$ associated to $\text{diff}F_\nu$ via the Riesz isomorphism, i.e.
\begin{equation}\label{grad-P}
	\nabla^\Pscr F_\nu := \sum_{i=1}^k \partial_i f(\nu(h_1),\dots,\nu(h_k)) \nabla h_i.
\end{equation}
Consequently, by definition we have
\begin{equation}\label{id-grad-diff-P}
	\text{diff}F_\nu (\varphi) = \langle \nabla ^\Pscr F_\nu, \varphi \rangle_\nu, \quad \forall \varphi \in T_\nu \Pscr = L^2(\R^d,\R^d;\nu).
\end{equation}
In particular, $\nabla^\Pscr F$ is independent of the representation $F(\mu) = f(\mu(h_1),\dots,\mu(h_k))$ of $F$.
For $G\in \Fscr C^2_b$, analogously to Section \ref{subsect:gradflow-general}, $\nabla^\Pscr G_\nu$ as a derivation acts on $F \in \Fscr C^2_b$ via
\begin{equation}\label{aux2}
	\nabla^\Pscr G_\nu(F) = \text{diff}G_\nu(\nabla^\Pscr F_\nu)= \langle \nabla^\Pscr G_\nu, \nabla^\Pscr F_\nu \rangle_\nu = \sum_{i=1}^k \partial_i f(\nu(h_1),\dots,\nu(h_k)) \langle \nabla h_i, \nabla^\Pscr G_\nu \rangle_\nu.
\end{equation}

The following slight generalizations have not been considered in \cite{RRW20} or, as far as we know, elsewhere, but are necessary for our main results.
\paragraph{Generalized differential and gradient.} More generally, for $G: D(G) \subseteq \Pscr \to \R$ not necessarily from $\Fscr C^2_b$, we define $D(\text{diff}G)$ to be the set of all $\nu \in D(G)$ such that
\begin{equation}\label{def:diffG-general}
		\text{diff}G_\nu (\varphi) := \ddt G(\mu_t^{\varphi, \nu})_{|t=0},\quad \varphi \in C^1_c(\R^d,\R^d),
\end{equation}
is well-defined and linear as well as continuous with respect to the usual $L^2(\R^d,\R^d;\nu)$-topology. For such $\nu$, $\text{diff}G_\nu$ has a unique linear and continuous extension to all of $L^2(\R^d,\R^d;\nu)$. Then again we define $\nabla^\Pscr G_{\nu}$ as the unique element in $T_\nu \Pscr$ such that
\begin{equation*}
	\text{diff}G_\nu(\varphi) = \big\langle \nabla^\Pscr G_\nu,\varphi \big\rangle_{\nu} ,\quad \forall\varphi \in C_c^1(\R^d,\R^d),
\end{equation*}
i.e. \eqref{aux2} remains valid.
In particular, we allow $G(\nu)$, $\text{diff}G_\nu$ and $\nabla^\Pscr G_\nu$ to be defined for $\nu$ from strict subsets of $\Pscr$ only.

\paragraph{Weighted metric tensors and gradient.}
We also introduce \textit{weighted} metric tensors on the tangent bundle $T\Pscr$, which are needed for our main results, in particular Theorem \ref{thm1}, as follows.
Let $\alpha: \Pscr \to \Bscr_b(\R^d)$ be such that $C^{-1} \leq \alpha(\nu)(x) \leq C$ for all $x \in \R^d$ and $\nu \in \Pscr$, with $C = C(\alpha) >1$ independent of $x$ and $\nu$. The inner product 
$$\langle \varphi, \tilde{\varphi}\rangle_{\alpha, \nu} := \langle \alpha(\nu)\varphi, \tilde{\varphi}\rangle_{\nu},\quad \varphi, \tilde{\varphi} \in L^2(\R^d,\R^d;\nu)$$ 
 is equivalent to the standard inner product $\langle \cdot, \cdot \rangle_{\nu}$, and we denote the weighted metric tensor $\nu \mapsto \langle \cdot, \cdot \rangle_{\alpha, \nu}$ by $\langle \cdot, \cdot \rangle_\alpha$. For $G: D(G)\subseteq \Pscr \to \R$ and $\nu\in D(\text{diff}G)$, we denote the gradient of $G$ at $\nu$ with respect to $\langle \cdot, \cdot \rangle_\alpha$ by $\nabla^\Pscr_\alpha G_\nu$, i.e. it is the unique element in $L^2(\R^d,\R^d;\nu)$ such that
\begin{equation*}
	\text{diff} G_\nu (\varphi) = \langle \nabla^\Pscr_\alpha G_\nu, \varphi \rangle_{\alpha,\nu},\quad \forall \varphi \in C^1_c(\R^d,\R^d).
\end{equation*}
For $G \in \Fscr C^2_b$, we have
\begin{equation}\label{eq:alpha-grad}
	\nabla^\Pscr_\alpha G_\nu = \alpha(\nu)^{-1}\nabla^\Pscr G_\nu.
\end{equation}

\subsection{Gradient flow character of nonlinear Fokker--Planck equations}\label{subsect:lifting-reveals}
Here, through a heuristic computation, we explain a general method to reveal the gradient flow character of nonlinear Fokker--Planck--Kolmogorov equations by considering solutions of the latter as curves on $\Pscr$ with the geometry from Section \ref{subsect:diff-geom}. First, we repeat the definition of distributional solutions to \eqref{intro:NLFPKE}.
\begin{dfn}\label{def:distr-sol-gen-NLFPKE}
	A \textit{distributional probability solution} $(0,\infty)\ni t\mapsto \mu_t$ to \eqref{intro:NLFPKE} is a Borel curve $t \mapsto \mu_t$ of probability measures $\mu_t \in \Pscr$ such that $(t,x)\mapsto a_{ij}(t,\mu_t,x)$ and $(t,x)\mapsto b_i(t,\mu_t,x)$, $i,j  \in \{1,\dots,d\},$ belong to $L^1_{\textup{loc}}([0,\infty)\times \R^d;\mu_tdt)$, and for all $0 \leq s <t$ and $\zeta \in C^2_c(\R^d)$
	\begin{equation*}
\int_{\R^d}\zeta(x)\,d\mu_t(x)-\int_{\R^d}\zeta(x)\,d\mu_s(x) = \int_s^t \int_{\R^d} a_{ij}(r,\mu_r,x)\partial_{ij}\zeta(x)+b_i(r,\mu_r,x)\partial_i \zeta(x)\,d\mu_r(x)dr.
			\end{equation*}
For brevity, we simply say \textit{solution} instead of \textit{distributional probability solution}.
\end{dfn}
Such solutions are clearly weakly continuous. Moreover, it follows that $t \mapsto \mu_t(\zeta)$ is differentiable $dt$-a.s. for every $\zeta \in C^2_c(\R^d)$, with derivative
	\begin{equation*}
	\ddt \int_{\R^d} \zeta(x) \,d\mu_t(x) = \int_{\R^d} a_{ij}(t,\mu_t,x)\partial_{ij}\zeta(x)  +b_i(t,\mu_t,x)\partial_i \zeta(x)\,d\mu_t(x)\quad dt\text{-a.s.},
\end{equation*}
where the exceptional set can be chosen independent of $\zeta$.
For the following heuristic computation we assume $\mu_t(dx) = u(t,x)dx$ and that the dependence of $a_{ij}$ on $\mu$ is of Nemytskii-type. For $F \in \Fscr C^2_b, F(\nu) = f(\nu(h_1),\dots,\nu(h_k))$, $t\mapsto F(\mu_t)$ is differentiable $dt$-a.s. with derivative
\begin{align*}
	\ddtau F(\mu_\tau)_{|\tau=t} &= \sum_{l=1}^k \partial_l f(\mu_t(h_1),\dots,\mu_t(h_k))\ddtau \mu_\tau(h_l)_{|\tau=t} \\&= \sum_{l=1}^k \partial_l f(\mu_t(h_1),\dots,\mu_t(h_k)) \int_{\R^d}a_{ij}(t,\mu_t)\partial_{ij} h_l + b(t,\mu_t)\nabla h_l \,d\mu_t
	\\& = \bigg\langle \sum_{l=1}^k \partial_lf(\mu_t(h_1),\dots,\mu_t(h_k)) \nabla h_l, \frac{-\divv (a(t,\mu_t)u(t))}{u(t)}+b(t,\mu_t) \bigg\rangle_{\mu_t} \\&= \bigg\langle \nabla^\Pscr F_{\mu_t}, \frac{-\divv (a(t,\mu_t)u(t))}{u(t)}+b(t,\mu_t)\bigg\rangle_{\mu_t}
\end{align*}
for $dt$-a.e. $t >0$.
Here we ignored the question of differentiability of $u(t)$ in $x$ and wrote $a = (a_{ij})_{1 \leq i,j \leq d}$ and $\divv (au) \in \R^d$ for the vector with entries $\partial_i(u a_{ij})$, $1\leq j \leq d$ (using Einstein summation convention).
Hence $t\mapsto \ddtau {\mu_\tau}_{|\tau=t}$ is a curve of tangent vectors (up to an $dt$-zero set) with action
\begin{equation*}\label{eq:aux1}
\ddtau {\mu_\tau}_{|\tau=t} (F)=  \text{diff}F_{\mu_t}(\ddtau {\mu_\tau}_{|\tau=t})=	\ddtau F(\mu_\tau)_{|\tau=t} =\bigg\langle \nabla^\Pscr F_{\mu_t}, \frac{-\divv( a(t,\mu_t)u(t))}{u(t)}+b(t,\mu_t)\bigg\rangle_{\mu_t} \quad \forall F \in \Fscr C^2_b.
\end{equation*}
Consequently, if there is $E: \Pscr \to \R$ such that $dt$-a.s. $\mu_t \in D(\text{diff}E)$ and
\begin{equation}\label{eq:ansatz-find-E}
	\text{diff}E_{\mu_t}(\nabla^\Pscr F_{\mu_t}) = \bigg\langle \nabla^\Pscr F_{\mu_t}, \frac{\divv( a(t,\mu_t)u(t))}{u(t)}-b(t,\mu_t)\bigg\rangle_{\mu_t},
\end{equation}
it follows that
\begin{equation*}
\nabla^\Pscr E_{\mu_t} = \frac{\divv( a(t,\mu_t)u(t))}{u(t)}-b(t,\mu_t) \in L^2(\R^d,\R^d;\mu_t)\quad dt\text{-a.s.},
\end{equation*}
and hence that $t\mapsto \mu_t$ satisfies $dt$-a.s. (compare with \eqref{GF-fct})
\begin{equation*}
	\ddtau {\mu_\tau}_{|\tau =t} = -\nabla^\Pscr E_{\mu_t}
\end{equation*}
in $T_{\mu_t}\Pscr$.
The $dt$-zero exceptional set cannot be avoided, since in general it is not possible to obtain differentiability of $t \mapsto \mu_t(\zeta)$ for \textit{every} $t>0$ from Definition \ref{def:distr-sol-gen-NLFPKE}.
\begin{rem}\label{rem:special-case-PME}
In the case $a_{ij}(t,\nu,x) = \delta_{ij}\frac{\beta(v(x))}{v(x)}$, $\nu(dx) = v(x)dx$, which is treated in Section \ref{sect:results}, one has $\frac{\divv(a(t,\mu_t)u(t))}{u(t)} = $$\frac{\nabla \beta(u(t))}{u(t)}$.
\end{rem}
Hence the essence of our ansatz is to first evaluate not only the linear maps $\nu \mapsto \nu(\zeta), \zeta\in C^2_c(\R^d)$, along solutions $t\mapsto \mu_t$ to \eqref{intro:NLFPKE}, but the much bigger class of nonlinear functions $F \in \Fscr C^2_b$, and then to identify the energy $E$ via \eqref{eq:ansatz-find-E}.

\subsection{Gradient flows on $\Pscr$}\label{subsect:gradflow-P}
 Let $\alpha: \Pscr \to \Bscr_b(\R^d), \, C^{-1}< \alpha < C, \, C>1$, be a weight as in Section \ref{subsect:diff-geom} with metric tensor $\langle \cdot, \cdot \rangle_\alpha$ (the non-weighted case corresponds to $\alpha \equiv 1$).

\begin{dfn}
	Let $I  = (a,b)$, $-\infty \leq a < b \leq +\infty$, and $t\mapsto \mu_t$ be such that there is a $dt$-zero set $N \subseteq I$ such that for all $t \in N^c$, $t\mapsto \mu_t(\zeta)$ is differentiable for all $\zeta\in C^2_c(\R^d)$. For $E: D(E) \subseteq \Pscr \to \mathbb{R}$ such that $\mu_t \in D(\text{diff}E)$ (and thus $\nabla^\Pscr_\alpha E_{\mu_t}$ is defined for all $t \in N^c$), $t\mapsto \mu_t$ is called a solution to the gradient flow with energy $E$ and weight $\alpha$, if it satisfies
	\begin{equation}\label{gradflow-P}\tag{$\Pscr$-GF}
		\frac{d}{d\tau} {\mu_\tau}_{|\tau=t} = -\nabla^\Pscr_\alpha E_{\mu_t}\quad dt\text{-a.s.}
	\end{equation}
\end{dfn}
\eqref{gradflow-P} is equivalent to $\mu_t \in D(\text{diff}E)$ $dt$-a.s. and 
\begin{equation}\label{GF-P-fct}
	\ddtau F\big(\mu_\tau\big)_{|\tau =t} = -\text{diff} E_{\mu_t}(\nabla^\Pscr_\alpha F_{\mu_t})\,\,\, \forall F \in \Fscr C^2_b\quad dt-\text{a.s.},
\end{equation}
and for $F(\nu) = f(\nu(h_1),\dots,\nu(h_k))$ and $ t\in N^c$, by \eqref{eq:alpha-grad} the latter is equivalent to
\begin{equation*}
	\sum_{i=1}^k\partial_if\big(\mu_t(h_1),\dots,\mu_t(h_k)\big){\ddtau\mu_\tau(h_i)}_{|\tau=t} =-\sum_{i=1}^k\partial_if\big(\mu_t(h_1),\dots,\mu_t(h_k)\big)\text{diff}E_{\mu_t}(\alpha(\mu_t)^{-1}\nabla h_i).
\end{equation*}
Therefore $t\mapsto \mu_t$ solves \eqref{gradflow-P} if and only if there is a zero set $N \subseteq I$ such that for all $t \in N^c$ and $\zeta \in C^2_c(\R^d)$ (with exceptional set independent of $\zeta$)
\begin{equation}\label{gradflow-P-single-h}
	{\ddtau\mu_\tau(\zeta)}_{|\tau=t} =  -\text{diff}E_{\mu_t}(\alpha(\mu_t)^{-1}\nabla \zeta).
\end{equation}
\begin{rem}\label{rem:decreasing-along-E}
	For $t\mapsto \mu_t$ and $E$ as above, heuristically choosing $F=E$ in \eqref{GF-P-fct} yields
	\begin{equation*}
		E(\mu_t)-E(\mu_s) = -\int_s^t |\nabla^\Pscr_\alpha E_{\mu_r}|^2_{\alpha,\mu_r}\,dr \leq 0,\quad \forall s < t \in I.
	\end{equation*}
However, since we do not necessarily have $E \in \Fscr C^2_b$ ($\nabla E_\alpha^\Pscr$ might be defined in the generalized sense explained in Section \ref{subsect:diff-geom}), the choice $F=E$ in \eqref{GF-P-fct} might not be permitted.
\end{rem}

\section{Generalized PME as gradient flows on $\Pscr$ and identification of the energy}\label{sect:results}

In this section we present our main results: We show that solutions to a class of generalized PMEs solve gradient flows equations on $\Pscr$ and we identify the corresponding energy function. We also prove a uniqueness result for the gradient flows. First, we consider the classical PME, see Theorem \ref{thm:classPME-case}. In Section \ref{sect:main-sect:general-case} we consider generalized PMEs with general diffusivity functions and an additional nonlinear transport-type first-order term, see Theorem \ref{thm1}. Finally, in Section \ref{sect:matrix-case}, we present a further generalization to a larger class of divergence-type equations.
We stress that throughout this section the uniqueness of solutions to the FPK equation does not play any role for our results and proofs, except, of course, for the uniqueness assertions concerning the gradient flow in theorems \ref{thm:classPME-case} and \ref{prop:main-appl-general-case}.
\\

Consider the equation \eqref{intro:genPME}. Conditions on the diffusivity function $\beta$ as well as on the spatial drift-vector field $D$ and the nonlinearity $b$ are given below. In the literature, equations with $D =0$ are also called generalized PME or PME with generalized diffusivity, see for instance \cite{V07}. In the present paper, \textit{generalized} PME refers both to the general diffusivity $\beta$ and to the additional transport-type drift. \eqref{intro:genPME} is a special case of \eqref{intro:NLFPKE} with $a_{ij}(t,\mu,x)= \delta_{ij}\frac{\beta(u(x))}{u(x)}$ and $b_i(t,\mu,x) = D^{(i)}(x)b(u(x))$, where $\mu(dx) = u(x)dx$.
For the case of Nemytskii-type coefficients, the following definition of solution is a more common (but equivalent) formulation than Definition \ref{def:distr-sol-gen-NLFPKE}. 

\begin{dfn}\label{d6.1}\rm $\mu:[0,{\infty})\to{\mathcal{P}}, \mu: t \mapsto \mu_t$, is a \textit{weakly continuous distributional probability solution to \eqref{intro:genPME}} (for brevity just \textit{solution}) with initial datum $\mu_0 \in \Pscr$ if $t \mapsto \mu_t$ is weakly continuous, $\mu_t(dx)= u(t,x)dx$ $dt$-a.s.,
	\begin{equation*}
		\label{e6.1}
		u, \beta(u)\in L^1_{\rm loc}([0,{\infty})\times{\mathbb{R}}^d), \,\, b(u)D \in L^1_{\rm loc}([0,{\infty})\times{\mathbb{R}}^d,\R^d), \end{equation*}
	and
	\begin{eqnarray*}
		\label{e6.2}
		&&\hspace*{-10mm}\displaystyle \int^{\infty}_0\int_{{\mathbb{R}}^d}
		u(t,x)(\partial_t\zeta(t,x)+b(u(t,x))D(x)\cdot\nabla{\zeta}(t,x))\\
		&&\hspace*{-10mm}\quad+\beta(u(t,x))\Delta{\zeta}(t,x)dt\,dx+\mu_0({\zeta}(0,\cdot))=0,
		\ \forall {\zeta}\in C^{\infty}_0([0,{\infty})\times{\mathbb{R}}^d).\nonumber\end{eqnarray*}
\end{dfn}
It follows that the initial datum $\mu_0$ is attained weakly, i.e. $\mu_t \longrightarrow \mu_0$ as $t \to 0$ in the topology of weak convergence of probability measures.

\subsection{The classical porous media equation}\label{subsect:classPME}
For $m>1$, consider the classical porous media equation \eqref{intro:classPME} on $\R^d$,
i.e. \eqref{intro:genPME} with $\beta(r) =| r|^{m-1}r$ and $D =0$, with an arbitrary initial condition $\mu_0 \in \Pscr$.
By \cite[Thm.1]{Pierre82}, for every $\mu_0 \in \Pscr$, there is a unique solution $(0,\infty) \ni t \mapsto u(t,x)dx \in \Pscr$ to \eqref{intro:classPME} in $\bigcap_{\delta>0}L^\infty((\delta,\infty)\times \R^d)$ such that $u(t,x)dx \longrightarrow \mu_0$ weakly as $t \to 0$. Furthermore, by \cite[Sect.5]{NLFPK-DDSDE5}, $u(t)^{\alpha} \in H^1(\R^d)$ for all $\alpha \geq \frac{m+1}{2} $ for $dt$-a.a. $t >0$.

\paragraph{Identification of the energy.}
According to Section \ref{subsect:lifting-reveals} (in particular, Remark \ref{rem:special-case-PME}), the ansatz to find the energy $E$ for the gradient flow equation is to find $E$ such that $\mu_t \in D(\text{diff}E)$ dt$-a.s.$ and
$$	\text{diff}E_{\nu}(\varphi) = \bigg\langle \varphi, \frac{\nabla (v^m)}{v}\bigg\rangle_{\nu} \quad \forall \varphi \in C^1_c(\R^d,\R^d)$$
for all $\nu \in D(\text{diff}E)$.
We shall obtain
$$D_{\text{nice}}:= \bigg\{\nu(dx) = v(x)dx \in \Pscr\,\big|\, v \in L^\infty(\R^d), v^m \in W^{1,1}_{\text{loc}}, \frac{\nabla(v^m)}{v}\in L^2(\R^d,\R^d;\nu)\bigg\} \subseteq D(\text{diff}E).$$
As shown in the proof of Theorem \ref{thm:classPME-case} below, the right choice of $E$ is
\begin{equation}\label{def:E-classPME}
	E: D(E)\subseteq \Pscr_a \to \R, \quad E(udx) := \int_{\R^d} \eta(u(x)) dx, \quad \eta(r):= m\int_0^r \int_1^s w^{m-2}\, dw \,ds, r\geq 0, 
\end{equation}
i.e. we have, since $D(E)\subseteq \Pscr$,
\begin{equation}\label{eq:E-class-PME}
	E(u dx ) = \frac{1}{m-1}\bigg(\int_{\R^d}u(x)^{m} dx -m\bigg),
\end{equation}
and
$D(E) = \Pscr_a \cap L^m(\R^d)$. Since $E$ appears only through its differential (a first-order functional), the (compared with Otto \cite{Otto01}) additional zero-order summand $\frac{-m}{m-1}$ can be dropped.

\paragraph{Main result.}
Note that for $\varphi \in C^1_c(\R^d,\R^d)$, $\nu(dx)= v(x)dx \in \Pscr_a$ and $|\tau| < \varepsilon = \varepsilon(\varphi) >0$, by the transformation rule the measure $\mu^{\varphi,\nu}_\tau = \nu \circ (\Id+\tau \varphi)^{-1}$ is absolutely continuous with respect to $dx$ with density
\begin{equation}\label{eq:density-pushforward}
	\frac{d\mu^{\varphi,\nu}_\tau}{dx}(x) = v\big((\Id+\tau \varphi)^{-1} (x)\big)|\det D(\Id+\tau\varphi)(x)|^{-1},
\end{equation}
where $D\psi$ denotes the Jacobian of $\psi \in C^1(\R^d,\R^d)$.

\begin{theorem}\label{thm:classPME-case}
	Let $m\geq 2$, $\mu_0 \in \Pscr$, and $t\mapsto \mu_t(dx) = u(t,x)dx$ the unique solution to \eqref{intro:classPME} in $\bigcap_{\delta >0}L^\infty((\delta,\infty)\times \R^d)$ with initial datum $\mu_0$.
	\begin{enumerate}
		\item [(i)] 	$t\mapsto \mu_t$ is a solution to \eqref{intro:GF} in $\Pscr$ with $E$ as in \eqref{def:E-classPME} and the standard $L^2$-metric tensor $\langle \cdot, \cdot \rangle$ on $T\Pscr$ (i.e. with weight $\alpha \equiv 1$). The gradient of $E$ for $\nu(dx) = v(x)dx \in D_{\text{nice}}$ is
		$$\nabla^\Pscr E_\nu=  \frac{\nabla (v^{m})}{v} \in L^2(\R^d,\R^d;\nu).$$
		\item [(ii)] If, in addition, $m\geq 3$, then
		\begin{equation*}
			\nabla^\Pscr E_{\mu_t}=  \frac{\nabla (u(t)^{m})}{u(t)} = \frac{m}{m-1}\nabla(u(t)^{m-1}) \quad dt\text{-a.s.},
		\end{equation*}
		i.e. in this case $\nabla^\Pscr E_{\mu_t}$  is a gradient vector field in $T_{\mu_t}\Pscr = L^2(\R^d,\R^d;\mu_t)$.
		\item[(iii)] $t\mapsto \mu_t$ is the unique solution to the gradient flow with $E$ as in $(i)$ in $\big(\bigcap_{\delta >0}L^\infty((\delta,\infty)\times \R^d)\big)\cap L^1_{\text{loc}}([0,\infty)\times \R^d)\cap D_{\text{nice}}$ such that $\mu_t \longrightarrow \mu_0$ weakly as $t\to 0$.
		\item [(iv)] If $d\geq 3$ and $\mu_0 \in L^\infty(\R^d)$, then $u \in L^\infty([0,\infty)\times \R^d)$ and for all $T>0$
		\begin{equation}\label{aux123}
			\int_0^T \int_{\R^d} |\nabla u(t)^{\frac{m+1}{2}}|^2\,dx dt < +\infty,
	\end{equation}
hence, if $m \geq 3$, 
\begin{equation}\label{aux1234}
	\int_0^T |\nabla^\Pscr E_{\mu_t}|^2_{T_{\mu_t}\Pscr}\,dt < +\infty.
\end{equation}
		
	\end{enumerate}

\end{theorem}

\begin{proof}
	\begin{enumerate}
		\item [(i)] Let $t \mapsto \mu_t(dx) = u(t,x)dx$ be as in the assertion. Then for all $\zeta \in C^2_c(\R^d)$ and $dt$-a.e. $t > 0$ (with exceptional set independent of $\zeta$), we have
		\begin{equation*}
			{\ddtau\mu_\tau(\zeta)}_{|\tau=t} = \int_{\R^d}\frac{-\nabla(u(t,x)^{m})}{u(t,x)}\cdot \nabla \zeta(x) \,d\mu_t(x).
		\end{equation*}
		Hence, considering \eqref{gradflow-P-single-h}, to prove the assertion, it is sufficient to show
		\begin{equation}\label{first-eq-proof1}
			D_{\text{nice}}\subseteq D(\text{diff}E),
		\end{equation}

		\begin{equation}\label{aux12}
			\int_{\R^d} \frac{-\nabla (v(x)^{m})}{v(x)} \cdot \varphi(x) \,d\nu(x) = -\text{diff}E_\nu(\varphi)
		\end{equation}
		for every $\varphi \in C^1_c(\R^d,\R^d)$ and $\nu(dx) = v(x)dx \in D_{\text{nice}}$, and $\mu_t \in  D_{\text{nice}}$ $dt$-a.s.
		Concerning \eqref{first-eq-proof1}, by \eqref{eq:density-pushforward} it is easy to see that $\mu^{\varphi,\nu}_\tau \in D(E)$ for all $\varphi \in C^1_c(\R^d,\R^d)$, $\nu(dx) = v(x)dx \in \Pscr_a$ such that $v \in L^\infty(\R^d)$ and $|\tau| < \varepsilon = \varepsilon(\varphi, \nu) >0$. 
		Moreover, $\tau \mapsto E(\mu^{\varphi,\nu}_\tau)$ is differentiable at $\tau =0$, since by the transformation rule
		\begin{align*}
			E(\mu^{\varphi,\nu}_\tau) &\notag =\int_{\R^d} \eta\bigg(\frac{d\mu^{\varphi,\nu}_\tau}{dx}(x) \bigg)\, dx 
			\\& \notag =\frac{1}{m-1}\int_{\R^d} v(x)^{m}|\det D(\Id+\tau\varphi)(x+\tau \varphi(x))|^{-m}|\cdot |\det D(\Id+\tau \varphi)(x)|\, dx - \frac{m}{m-1}
		\end{align*}
		and since by Lemma \ref{lem:det-calc} below and because $v \in L^\infty(\R^d)$ the integrand on the right is differentiable in $\tau \in (-\varepsilon,\varepsilon)$ for all $x \in \R^d$ with uniformly in $\tau$ $L^1(\R^d)$-bounded derivative. 
		Hence, for $\nu\in D_{\text{nice}}$, Lemma \ref{lem:det-calc} yields $\nu \in D(\text{diff}E)$ and
		\begin{equation*}
			\text{diff}E_{\nu}(\varphi)  = \ddtau E(\mu^{\varphi,\nu}_{\tau})_{|\tau=0}= -\int_{\R^d} v(x)^{m}\divv \varphi (x)\,dx  = \bigg\langle \frac{\nabla(v^m)}{v},\varphi \bigg\rangle_\nu.
		\end{equation*}
		Consequently, also \eqref{aux12} holds. It remains to show $\mu_t \in D_{\text{nice}}$ $dt$-a.s. To this end, recall $u(t)^{m-\frac 1 2} \in H^1(\R^d)$ $dt$-a.s. This follows from \cite[Rem.5.4]{NLFPK-DDSDE5} since $m-\frac 1 2 \geq \frac{m+1}{2}$ for $m \geq 2$. Note that \cite[Rem.5.4]{NLFPK-DDSDE5} is restricted to $d \geq 3$. However, this restriction is only needed for the more general diffusivity functions $\beta$ instead of $r\mapsto |r|^{m-1}r$ considered in \cite{NLFPK-DDSDE5}. More precisely, tracing through the proof of \cite[Thm.5.2]{NLFPK-DDSDE5}, one sees that $d \geq 3$ is only needed for the $L^1-L^\infty$-regularization result of Theorem 4.1 of the same reference. But this result is true for \eqref{intro:classPME} for any $d \geq 1$, see for instance \cite[Thm.9.12]{V07}. Since $u(t)^{m-\frac 1 2} \in H^1(\R^d)$ implies $u(t)^m \in W^{1,1}_{\text{loc}}(\R^d)$ as well as $\frac{\nabla(u(t)^m)}{u(t)}\in L^2(\R^d,\R^d;\mu_t)$, the proof is complete.
		\item[(ii)] For $m\geq 3$ we have $m-1 \geq \frac{m+1}{2}$, hence by \cite[Rem.5.4]{BR-IndianaJ}, $u(t)^{m-1} \in H^1(\R^d)$ $dt$-a.s. Now the additional equality of the assertion follows from a straightforward calculation.
		\item[(iii)] Any such solution $t\mapsto \nu_t$ to \eqref{intro:GF} with $E$ as in (i) is a solution to \eqref{intro:classPME}. Hence the claim follows from the restricted uniqueness result \cite[Thm.1]{Pierre82} for \eqref{intro:classPME}.
		\item[(iv)] The fact that, in this case, $u \in L^\infty([0,\infty)\times \R^d)$ is well-known (see e.g. \cite[Thm.2.2]{NLFPK-DDSDE5}). Hence by the uniqueness result in \cite{BR23} (see Remark 3.3 therein) and \cite[Thm.1.2(3),(4)]{RW08} (for the case $B \equiv 0$, i.e. zero noise) $u =X$, where $X$ is the solution to \eqref{intro:classPME} constructed in \cite{RW08} on $H^{-1}(\R^d)$ instead of $L^1(\R^d)$. Hence we may apply \cite[Thm.1.2(4)]{RW08} to prove \eqref{aux123}, which then implies \eqref{aux1234}.
	\end{enumerate}
\end{proof}
We used the following lemma for the previous proof. Recall that $D\psi$ denotes the Jacobian for $\psi \in C^1(\R^d,\R^d)$.
\begin{lem}\label{lem:det-calc}
	Let $\varphi \in C^1_b(\R^d,\R^d)$. Then, there is $\varepsilon = \varepsilon(\varphi)>0$, which does not depend on $x \in \R^d$, such that for all $x \in \R^d$:
	\begin{enumerate}
		\item [(i)] $\tau \mapsto \det D(\Id+\tau \varphi)(x)$ is differentiable on $(-\varepsilon,\varepsilon)$ with derivative $\divv \varphi(x)$ at $\tau =0$. 
		\item[(ii)] $\tau \mapsto \det D(\Id+\tau \varphi)(x+\tau \varphi(x))^{-1}$ is differentiable on $(-\varepsilon,\varepsilon)$ with derivative $-\divv \varphi(x)$ at $\tau =0$.
	\end{enumerate}
	
	Moreover, both maps and their first derivatives in $\tau$ are uniformly bounded in $(\tau,x) \in (-\varepsilon,\varepsilon)\times \R^d$.
\end{lem}
\begin{proof}
	Choose $\varepsilon>0$ such that $\det D(\Id+\tau \varphi)(x)>0$ for all $(t,x) \in (-\varepsilon,\varepsilon)\times \R^d$. This is possible, since $\varphi \in C^1_b(\R^d,\R^d)$, $\det D(\Id)(x)=1$ and since $(t,x)\mapsto \det D(\Id+\tau \varphi)(x)$ is continuous.
	\begin{enumerate}
		\item[(i)] The differentiability follows directly from the Leibniz formula. Concerning the derivative at $\tau=0$, note that in dimension $d=1$ one has
		\begin{equation*}
			\ddtau \big(\det D(\Id+\tau \varphi) (x)\big)_{|\tau=0} = \ddtau \big(1+\tau \varphi'(x)\big)_{|\tau=0} = \varphi'(x) = \divv \varphi (x).
		\end{equation*}
		Now the claim follows by induction over $d$ and by using the Laplace expansion for the determinant of a quadratic matrix. Indeed, regrouping all summands in the Laplace expansion in terms of their order in $t$, the only summand of order one is $\divv \varphi(x)$. Hence, by differentiating in $t$, the claim follows.
		\item[(ii)] Consider the map $(t,y) \mapsto \det D(\Id+\tau \varphi)(y)$. Again appealing to the Leibniz formula, clearly $(\tau,y)\mapsto \det D(\Id+\tau \varphi)(y)$ is differentiable on $(-\varepsilon,\varepsilon)\times \R^d$. Hence, for $x \in \R^d$ fixed (but arbitrary), (i) implies
		\begin{align*}
			\ddtau \big(\det D(\Id+\tau \varphi)(x+\tau \varphi(x))\big)^{-1}_{|\tau=0} &= -(1,\divv \varphi(x))\cdot (\divv \varphi (x),(\partial_y \det D(\Id+\tau \varphi)(y))_{|(\tau,y) = (0,x)})\\& = -\divv \varphi(x),
		\end{align*}
		since $(\partial_y \det D(\Id+\tau \varphi)(y))_{|(\tau,y)} = 0$. Indeed, by the Leibniz formula, $\partial_y \det D(\Id+\tau \varphi)(y)$ consists of summands of order of at least $1$ in $t$, and hence its evaluation at any point $(0,y)$ equals $0$.
	\end{enumerate}
	The final claim follows directly from the Leibniz formula.
\end{proof}

\begin{rem}\label{rem:class-PME-extend-m-less-2}
	
	The restriction $m \geq 2$ in Theorem \ref{thm:classPME-case} is only needed to apply \cite[Rem.5.4]{NLFPK-DDSDE5} in order to obtain $\frac{\nabla (u(t)^{m})}{u(t)}\in L^2(\R^d,\R^d;\mu_t)$. Without this, we cannot prove $\nabla^\Pscr E_{\mu_t} \in L^2(\R^d,\R^d;\mu_t) = T_{\mu_t}\Pscr$, i.e. we cannot prove that $\nabla^\Pscr E$ is a section along $t \mapsto \mu_t$ in the tangent bundle $T\Pscr$. But even for $m \in (1,2)$ we have $\nabla (u(t)^{m}) \in L^2(\R^d,\R^d)$ and thus $\frac{\nabla (u(t)^{m})}{u(t)} \in L_{\textup{loc}}^1(\R^d,\R^d;\mu_t)$. Then $\nabla^\Pscr E_{\mu_t}$ can be considered as the unique representing element in $L^1_{\textup{loc}}(\R^d,\R^d;\mu_t)$ of $\textup{diff}E_{\mu_t}$ on $(C^1_c(\R^d,\R^d),|\cdot|_\infty)$, and solutions to \eqref{intro:classPME} can be understood as solutions to the gradient flow in such a generalized sense.
\end{rem}

\begin{rem}
	For $d \geq 3$, the assertion and the proof of Proposition \ref{thm:classPME-case} can be extended to the case of a more general diffusivity function $\beta:r\mapsto \beta(r)$ in place of $r\mapsto |r|^{m-1}r$ in \eqref{intro:classPME}, with the following assumptions on $\beta$.
	\begin{equation*}
		\beta \in C^2(\R), \,\,\beta(0) = 0, \,\,|\beta(r)| \leq C r^\alpha, \,\,\beta'(r) \geq C|r|^{\alpha -1}
	\end{equation*}
	for $C >0$, $\alpha \geq 1$, and $\beta'$ is such that $\eta$ given below belongs to $C^1(\R)$ (which is, e.g., the case, if we have $\beta'(r) \leq C|r|^{\alpha-1}$ for some $\alpha \geq 2$). Then the energy $E$ is given by
	\begin{equation*}
		E: D(E)\subseteq \Pscr_a \to \R, \quad E(udx) := \int_{\R^d} \eta(u(x)) dx, \quad \eta(r):= \int_0^r \int_1^s \frac{\beta'(w)}{w}dw ds, r \geq 0,
	\end{equation*}
	where 
	$$D(E):= \{\nu(dx) = v(x)dx: \eta(v) \in L^1(\R^d)\}.$$
	Indeed, in this case, by \cite[Sect.5]{NLFPK-DDSDE5}, for any $\mu_0 \in \Pscr$ there is a solution $t \mapsto \mu_t(dx) = u(t,x)dx$ to \eqref{intro:classPME} (with $r\mapsto \beta(r)$ replacing $r\mapsto |r|^{m-1}r$) in $\bigcap_{\delta >0}L^\infty((\delta,\infty)\times \R^d)$ with initial datum $\mu_0$. However, here we need to assume additionally $\mu_t \in D(E)$ $dt$-a.s (which again follows if, e.g., $\beta'(r)\leq C|r|^{\alpha-1}$ for some $\alpha \geq 2$). Then a similar calculation as in the proof of Theorem \ref{thm:classPME-case} shows
	\begin{equation*}
		\text{diff}E_{\mu_t}(\varphi)  = \ddt E(\mu^{\varphi,\nu}_{\tau})_{|\tau=0}= \int_{\R^d} \big(\eta(u(t,x))-\eta'(u(t,x))u(t,x)\big)\divv \varphi (x)\,dx,\quad \forall \varphi \in C^1_c(\R^d,\R^d).
	\end{equation*}
	Since $\big(\eta(r)-\eta'(r)r\big)' = -\beta'(r)$ and since \cite[Rem.5.4]{NLFPK-DDSDE5} yields $\beta(u(t)) \in H^1(\R^d)$ $dt$-a.s., we have $dt$-a.s.
	\begin{equation*}
		\text{diff}E_{\mu_t}(\varphi) = \int_{\R^d} \nabla(\beta(u(t,x))) \cdot\varphi (x)\,dx.
	\end{equation*}
	Hence, as in Remark \ref{rem:class-PME-extend-m-less-2}, $\nabla^\Pscr E_{\mu_t} = -\frac{\nabla (\beta(u(t)))}{u(t)}$ is the unique $L_{\textup{loc}}^1(\R^d,\R^d;\mu_t)$-element representing $\textup{diff}E_{\mu_t}$ on $(C^1_c(\R^d,\R^d),|\cdot|_\infty)$, and $t\mapsto \mu_t$ can be understood as a solution to the gradient flow in this generalized sense.
\end{rem}

\subsection{Generalized porous media equation}\label{sect:main-sect:general-case}
Now we consider equation \eqref{intro:genPME}. We will prove that solutions to this equation solve \eqref{gradflow-P}
 with a weighted metric tensor $\langle \cdot, \cdot \rangle_\alpha$ (see Section \ref{subsect:diff-geom}). Consider the following set of assumptions on the coefficients $\beta, b$ and $D$.
\\

\noindent \textbf{Hypothesis 1}
\begin{itemize}
	\item[(i)] $\beta\in C^1(\R),\ \beta(0)=0,\ \gamma\le\beta'(r),\
	r\in\R,$ for $0<\gamma<\infty.  $
	\item[(i)'] $\beta\in C^1(\R),\ \beta(0)=0,\ \gamma\le\beta'(r)\le\gamma_1,\
	 r\in\R,$ for $0<\gamma<\gamma_1<\infty.  $
	\item[(ii)] $b\in C_b(\R)\cap C^1(\R), b \geq b_0 >0$.
	\item [(iii)] $\Phi \in C^1(\R^d)$, $\nabla \Phi \in C_b(\R^d,\R^d)$, $D = -\nabla \Phi$.
		\item [(iv)] $(\divv D)^- \in L^\infty(\R^d)$ and ${\rm div}\,D\in  (L^2(\R^d)+L^\infty(\R^d))$.	
	\item[(v)]$\Phi \in W^{2,1}_{\rm loc}(\R^d),\ \Phi\ge1,$ $\lim\limits_{|x|\to\infty}\Phi(x)=+\infty$ and there exists $m\in[2,\infty)$ such that $\Phi^{-m}\in L^1(\R^d)$.
\end{itemize}
First, we assume (i), (ii), (iii). The ansatz for the following energy functional comes again from Section \ref{subsect:lifting-reveals}. Set 
\begin{equation}\label{def:eta}
	\eta: \R_+ \to \R, \quad \eta(r):= \int_0^r g(s)ds := \int_0^r \int_1^s\frac{\beta'(w)}{wb(w)}dw \, ds,
\end{equation}
and
\begin{equation}\label{def:E}
	E: \tilde{D}(E)\subseteq \Pscr_a \to \R \cup \{+\infty\}, \quad E(vdx):= \int_{\R^d} \eta(v(x))\, dx + \int_{\R^d} \Phi(x)\,v(x)dx,
\end{equation}
where $\tilde{D}(E) := \{\nu(dx) = v(x)dx  \in \Pscr_a \,|\,\Phi \in L^1(\R^d;\nu)\}$. 
Note that
by (i) and (ii)
we have for $r \geq0$ 
\begin{equation*}
	\frac{\gamma_1}{b_0}\,\mathds{1}_{[0,1]}(r)r(\log r-1)
	+\frac{\gamma}{|b|_\infty}\,\mathds{1}_{(1,\infty)}(r)
	r(\log r-1)\le\eta(r)
	\le \frac{\gamma}{|b|_\infty}\,\mathds{1}_{[0,1]}(r)r(\log r-1)+
	\frac{\gamma_1}{b_0}\,\mathds{1}_{(1,\infty)}(r)r(\log r-1),
\end{equation*}
which yields the second equality in the next line
\begin{equation*}
	D(E)= \tilde{D}(E)\cap \{\nu(dx) = v(x)dx \in \Pscr_a\,|\,E(v(x)dx) < +\infty\} = \tilde{D}(E)\cap \{\nu(dx) = v(x)dx \in \Pscr_a\,|\, v \log v \in L^1(\R^d)\}.
\end{equation*}
For the main results of this section, Theorems \ref{thm1} and \ref{prop:main-appl-general-case}, we consider the weight $\alpha$,
\begin{equation}\label{def:weighted-tensor1}
	\alpha: \Pscr \to \Bscr_b(\R^d), \quad \alpha (\nu):= \begin{cases}
		x\mapsto \frac{1}{b(v(x))}&,\text{ if }\nu \in \Pscr_a, \nu(dx)= v(x)dx,
		\\1&,\text{ else,}
		\end{cases}
\end{equation}
 and we denote the corresponding metric tensor and gradient by $\langle \cdot, \cdot \rangle_b$ and $\nabla^\Pscr_b$, i.e. 
\begin{equation}\label{def:weighted-tensor2}
\langle \cdot, \cdot \rangle_b : \nu \mapsto \langle\cdot,\cdot \rangle_{b,\nu} = \bigg\langle \frac{1}{b(v(x))}\cdot, \cdot \bigg\rangle_{\nu}
\end{equation} 
for $\nu(dx) = v(x)dx \in \Pscr_a$, and $\langle \cdot, \cdot \rangle_{b,\nu} = \langle \cdot, \cdot \rangle_{\nu}$ if $\nu \in \Pscr \backslash \Pscr_a$. For $G: D(G)\subseteq \Pscr \to \R$ and $\nu \in D(\text{diff}G)$, $\nabla^\Pscr_bF_{\nu}$ denotes the unique element in $L^2(\R^d,\R^d;\nu)$ such that $\text{diff} F_\nu(\varphi) = \langle \nabla^{\Pscr}_bF_{\nu},\varphi \rangle_{b,\nu}$ for all $\varphi \in C^1_c(\R^d,\R^d)$.
We introduce the set
$$D'_{\text{nice}}:= \bigg\{\nu(dx) = v(x)dx \in \Pscr \,|\, v \in L^\infty(\R^d)\cap W^{1,1}_{\text{loc}}(\R^d), \frac{\nabla v}{ v} \in L^2(\R^d,\R^d;\nu)\bigg\},$$
and we will prove $D'_{\text{nice}} \cap D(E)\subseteq D(\text{diff}E)$ under the assumptions of Hypothesis 1 (see Theorem \ref{prop:main-appl-general-case}).

\begin{rem}
	For the linear special case of \eqref{intro:genPME} being the heat equation, i.e. $\beta(r)=r$ and $D= b= 0$, one has 
	$$D(E) = \big\{v(x)dx \in \Pscr_a \,|\, v \log v \in L^1(\R^d)\big\},\quad E(vdx) = \int_{\R^d} v(x)(\log v(x)-1) \,dx.$$
	Hence, in this case, $E$ is just the classical Boltzmann entropy function. The same energy function $E$ for the heat equation was also formally obtained in \cite{Otto01}.
\end{rem}

\begin{theorem}\label{thm1}
Suppose (i)', (ii), (iii) of Hypothesis 1 hold. Set $D_0:= D'_{\text{nice}}\cap D(E)$. Let	$ t\mapsto \mu_t(dx) = u(t,x)dx$ be a solution to \eqref{intro:genPME} with initial datum $\mu_0 \in \Pscr$ such that $\mu_t \in D_0$ for $dt$-a.e. $t >0$.
Then $t\mapsto \mu_t$ solves \eqref{gradflow-P} with $E$ as in \eqref{def:E} and with the weighted metric tensor $\langle \cdot, \cdot \rangle_b$ from \eqref{def:weighted-tensor2}. We have $D_0\subseteq D(\text{diff}E)$ and the weighted gradient of $E$ for $\nu(dx) = v(x)dx \in D_0$
is
\begin{equation}\label{aux100}
	\nabla^\Pscr_b E_\nu = \frac{\nabla(\beta(v))}{v}-b(v)D = b(v)\nabla\big(g(v)+\Phi\big)
\end{equation}
in $(L^2(\R^d,\R^d;\nu),\langle \cdot, \cdot \rangle_{b,\nu})$.
\end{theorem}
\begin{proof}
For $t \mapsto \mu_t(dx) = u(t,x)dx$ as in the assertion we have for all $\zeta \in C^2_c(\R^d)$ and $dt$-a.e. $t > 0$ (with exceptional set independent of $\zeta$)
\begin{equation*}
	{\ddtau\mu_\tau(\zeta)}_{|\tau=t} = \int_{\R^d}\frac{\beta(u(t,x))}{u(t,x)}\Delta \zeta +D(x)b(u(t,x))\cdot \nabla \zeta \,d\mu_t(x).
\end{equation*}
Hence, by Section \ref{subsect:gradflow-P}, it remains to prove
\begin{equation}\label{aux11}
D'_{\text{nice}}\cap D(E)	\subseteq D(\text{diff}E),
\end{equation}
\begin{equation}\label{proof:show-this-then-done}
-\text{diff}E_{\nu}(b(v) \varphi) = -\int_{\R^d}\bigg(\frac{\nabla \beta(v(x))}{v(x)} -D(x)b(v(x))\bigg)\cdot \varphi\,d\nu(x)
\end{equation}
for every $\varphi \in C^1_c(\R^d,\R^d)$ and $\nu \in v(x)dx \in D'_{\text{nice}}\cap D(E)$. 
Concerning \eqref{aux11}, let $\nu =v(x)dx \in D'_{\text{nice}}\cap D(E)$. Recall that for $\varphi \in C^1_c(\R^d,\R^d)$ and $|\tau| < \varepsilon = \varepsilon(\varphi) >0$, $\mu^{\varphi,\nu}_\tau = \nu \circ (\Id+\tau \varphi)^{-1}$ is absolutely continuous with density as in \eqref{eq:density-pushforward}.
Moreover, $\mu^{\varphi, \nu}_\tau \in D(E)$ for $|\tau|<\varepsilon$ (with $\varepsilon>0$ as above). Indeed,
	\begin{align*}
		\int_{\R^d}|\Phi(x+\tau \varphi(x))| v(x)dx = &\int_{(\supp \varphi)^c}|\Phi(x)| v(x)dx \\&+ \int_{(\Id+\tau \varphi)(\supp \varphi)}|\Phi(x)| v(\Id+\tau \varphi)^{-1}(x))|\det D(\Id+\tau \varphi)(x)|^{-1} dx,
	\end{align*}
and both summands on the right-hand side are finite, since $v\in D(E) \cap L^\infty(\R^d)$, $\Phi \in L^1_{\text{loc}}(\R^d)$, and due to Lemma \ref{lem:det-calc}. Hence $\mu^{\varphi, \nu}_\tau \in \tilde{D}(E)$, so by the transformation rule it remains to show, abbreviating $j(\tau,x):= |\det D(\Id+\tau \varphi)(x+\tau \varphi (x))|^{-1}$, that $v(x)j(\tau,x)\log \big(v(x)j(\tau,x)\big) \in L^1(\R^d)$. But this follows, since $j(\tau,x)$ is uniformly in $(\tau,x) \in (-\varepsilon,\varepsilon)\times \R^d$ contained in an interval $(1-\delta, 1+\delta)$, where (after decreasing $\varepsilon$ if necessary) $\delta = \delta(\varepsilon) \in (0,1/2)$, and hence
\begin{align*}
	\int_{\R^d} v(x)j(\tau,x)|\log \big(v(x)j(\tau,x)\big)| \,dx \leq 2\int_{\R^d} v(x) (|\log v(x)| + \log 2)\,dx <\infty.
\end{align*}
Moreover, $\tau \mapsto E(\mu^{\varphi,\nu}_\tau)$ is differentiable in $\tau =0$, since by the transformation rule
\begin{align*}
E(\mu^{\varphi,\nu}_\tau) &\notag =\bigg[\int_{\R^d} \eta\bigg(\frac{d\mu^{\varphi,\nu}_\tau}{dx}(x) \bigg)\, dx + \int_{\R^d} \Phi(x)\,\bigg(\frac{d\mu^{\varphi,\nu}_\tau}{dx}(x) \bigg)dx\bigg]
	\\& \notag =\bigg[\int_{\R^d} \eta\bigg(v(x)|\det D(\Id+\tau\varphi)(x+\tau \varphi(x))|^{-1}\bigg)|\det D(\Id+\tau \varphi)(x)|\, dx\\& \notag \quad \quad \quad \quad \quad \quad \quad  + \int_{\R^d} \Phi(x+\tau \varphi(x))\,d\nu(x)\bigg]
	\\& =: I_1(\tau)+I_2(\tau),
\end{align*}
and the integrands of $I_i(\tau), i \in \{1,2\}$, are differentiable in $\tau \in (-\varepsilon,\varepsilon)$ for all $x \in \R^d$ with derivative uniformly $L^1(\R^d)$-, and $L^1(\R^d;\nu)$-bounded in $\tau$, respectively. Indeed, for $I_2$, this follows immediately from $\Phi \in C^1(\R^d)$ and the boundedness of $\nabla \Phi$ and $\varphi$. For $I_1$, the claim follows by Lemma \ref{lem:det-calc}, $v \in L^\infty(\R^d)$ and the local boundedness of $\eta'$.
Therefore, we obtain
\begin{align}\label{calc-long-with-regu}
	\notag	\text{diff}E_{\nu}(\varphi) &= \ddtau E(\mu^{\varphi,\nu}_\tau)_{|\tau =0} \notag =\ddtau \bigg[\int_{\R^d} \eta\bigg(\frac{d\mu^{\varphi,\nu}_\tau}{dx}(x) \bigg)\, dx + \int_{\R^d} \Phi(x)\frac{d\mu^{\varphi,\nu}_\tau}{dx}(x) dx\bigg]_{|\tau =0}
	\\& \notag =\ddtau \bigg[\int_{\R^d} \eta\bigg(v(x)|\det D(\Id+\tau\varphi)(x+\tau \varphi(x))|^{-1}\bigg)|\det D(\Id+\tau \varphi)(x)|\, dx\\& \quad \quad \quad \quad \quad \quad \quad  + \int_{\R^d} \Phi(x+\tau \varphi(x))\,d\nu(x)\bigg]_{|\tau =0}.
\end{align}
First note 
$$\ddtau\bigg[\int_{\R^d} \Phi(x+\tau \varphi(x))\,d\nu(x)\bigg]_{|\tau=0} = \int_{\R^d}\nabla \Phi \cdot \varphi (x)\,d\nu(x) = -\int_{\R^d}D(x) \cdot \varphi (x)\,d\nu(x).$$
Concerning the other summand in \eqref{calc-long-with-regu}, we find
\begin{align*}
&\quad\,\,\ddtau \bigg[\int_{\R^d} \eta\bigg(v(x)|\det D(\Id+\tau\varphi)(x+\tau \varphi(x))|^{-1}\bigg)|\det D(\Id+\tau \varphi)(x)|\, dx\bigg]_{|\tau =0}
\\&= \int_{\R^d} \eta'(v(x)) \bigg(\ddtau \bigg|\det D(\Id+\tau\varphi)(x+\tau \varphi(x))\bigg|^{-1}\bigg)_{|\tau =0}v(x)\\&\quad \quad \quad \quad  +\eta(v(x))\ddtau |\det D(\Id+\tau \varphi)(x)|_{|\tau =0} \,dx 
\\& = \int_{\R^d} \bigg(\eta (v(x))-\eta'(v(x))v(x)\bigg)\divv \varphi (x)\,dx,
\end{align*}
where we used Lemma \ref{lem:det-calc} for the final equality.
Since $v \in W^{1,1}_{\text{loc}}(\R^d)$ and since
$\eta (r)-\eta'(r)r$ is differentiable  with derivative
$= -\frac{\beta'(r)}{b(r)},$
by the divergence theorem we obtain
\begin{align*}
\int_{\R^d} \bigg(\eta (v(x))-\eta'(v(x))v(x)\bigg)\divv \varphi (x)\,dx = 	\int_{\R^d} \frac{\beta'(v(x))\nabla v(x)}{b(v(x))}\cdot \varphi \,dx  = \int_{\R^d} \frac{\beta'(v(x))\nabla v(x)}{v(x)b(v(x))}\cdot \varphi \,d\nu(x).
\end{align*}
Hence, since $\frac{\beta'(v)\nabla v}{ v} \in L^2(\R^d,\R^d;\nu)$, the map
\begin{equation*}
\varphi \mapsto	\text{diff}E_{\nu}(\varphi)  =  \int_{\R^d} \bigg(\frac{\beta'(v(x))\nabla v(x)}{v(x)b(v(x))}-D(x)\bigg)\cdot \varphi(x)\,d\nu(x)
\end{equation*}
is linear and continuous on $C^1_c(\R^d,\R^d)$ with respect to the $L^2(\R^d,\R^d;\nu)$-topology,
Thus we have shown \eqref{aux11}-\eqref{proof:show-this-then-done} and the proof is complete.
\end{proof}

Now assume (i) and (ii)-(v) of Hypothesis 1. In this case, \eqref{intro:genPME} has a unique bounded solution $t\mapsto u(t,x)dx$ for every initial datum $\mu_0 \in L^\infty(\R^d)\cap D(E)$, and, as we will show, this solution belongs to $D'_{\text{nice}}\cap D(E)$. Consequently, Theorem \ref{thm1} applies. More precisely, we have Theorem \ref{prop:main-appl-general-case} below. 

\begin{rem}
We point out that the assumption $\beta'\leq \gamma_1$ from (i)' of Hypothesis 1 is not needed here, since $\beta'(u)$ is bounded due to the boundedness of $u$, even if $\beta'$ is only locally bounded. Moreover, in this case we obtain a uniqueness result for the gradient flow, comparable to Theorem \ref{thm:classPME-case}, since for solutions in $L^1_{\text{loc}}([0,\infty)\times \R^d)$ all integrability conditions from Definition \ref{d6.1} are fulfilled. We stress that for the case of bounded $u$, Theorem \ref{thm1} also holds when (i)' is assumed instead of (i).
\end{rem}

\begin{theorem}\label{prop:main-appl-general-case}
	Suppose (i) and (ii)-(v) of Hypothesis 1 are satisfied and let $\mu_0 \in L^\infty(\R^d)\cap D(E)$.
	\begin{enumerate}
		\item [(i)] There is a unique solution to $t\mapsto \mu_t(dx) = u(t,x)dx$ \eqref{intro:genPME} in $L^\infty([0,\infty)\times \R^d$ with initial datum $\mu_0$, and it solves \eqref{gradflow-P} with $E$ from \eqref{def:E} and weighted metric tensor $\langle \cdot, \cdot \rangle_b$. We have $D'_{\text{nice}}\cap D(E)\subseteq D(\text{diff}E)$, and the formula for $\nabla_b^\Pscr E_\nu$ for $\nu \in D'_{\text{nice}}\cap D(E)$ is as in \eqref{aux100}.
		\item [(ii)] Moreover, the solution from (i) is the unique solution to this gradient flow in $L^\infty([0,\infty)\times \R^d) \cap D'_{\text{nice}}\cap D(E)$ with initial datum $\mu_0$.
		\item[(iii)] If $D = b= 0$ (i.e. $\langle \cdot, \cdot \rangle_{b}$ is the standard $L^2$-metric tensor $\langle \cdot, \cdot \rangle$), $t\mapsto \mu_t$ as in (i) satisfies
		\begin{equation*}
			\int_0^T |\nabla^\Pscr E_{\mu_t}|^2_{T_{\mu_t}\Pscr}\,dt < +\infty,\quad \forall T>0.
		\end{equation*}
	\end{enumerate}

\end{theorem}

\begin{proof}
	\begin{enumerate}
		\item [(i)] 	The existence assertion follows from \cite[Prop.2.2]{BR-IndianaJ} (there it is proven that such solutions exist for every $\mu_0 \in L^1(\R^d)\cap L^\infty(\R^d)$). The uniqueness follows from \cite[Cor.3.5]{BR23}. Hence it remains to prove $\mu_t \in D'_{\text{nice}}\cap D(E)$ $dt$-a.s. for this solution $t\mapsto \mu_t(dx) = u(t,x)dx$. First, $\mu_t \in D(E)\cap W^{1,1}_{\text{loc}}(\R^d)$ $dt$-a.s. follows from \cite[Thm.4.1]{BR-IndianaJ}. Furthermore, by \cite[Thm.4.1.,Eq.(4.7)]{BR-IndianaJ} and since $b(u(t))D \in L^\infty(\R^d,\R^d)$, the triangle inequality and $u(t) \in L^1(\R^d)$ yield $\beta'(u(t))\nabla u(t) (u(t))^{-1} \in L^2(\R^d,\R^d;\mu_t)$ $dt$-a.s.
		\item[(ii)] It is easily seen that any such solution to the gradient flow is a solution to the FPK equation. Hence the claim follows from the restricted uniqueness result in (i).
		\item[(iii)] This is proven analogously to Theorem \ref{thm:classPME-case} (iv).
	\end{enumerate}
\end{proof}

\begin{rem}[Stationary solutions to \eqref{intro:genPME}]
If in addition to Hypothesis 1 also the "balance condition"
	\begin{equation}
	\gamma_1 \Delta \Phi(x)-b_0|\nabla \Phi(x)|^2 \leq 0, \quad dx-\text{a.s.},
	\end{equation}
holds,
	 it was shown in \cite{BR-IndianaJ} that $E$ from \eqref{def:E} is a Lyapunov function for the solution $t \mapsto \mu_t$, i.e. $E(\mu_t)\leq E(\mu_s)$ for all $ 0 \leq s \leq t$. Moreover, it is proven there that $\mu_t$ converges to an equilibrium $u_\infty$ in $L^1(\R^d)$ as $t \to \infty$, which is a stationary solution to \eqref{intro:genPME} and can be calculated from $E$ via
	\begin{equation*}
		u_\infty(x) = g^{-1}\big(-\Phi + c\big),
	\end{equation*}
where $g(r) = \eta'(r)$ with $\eta$ from \eqref{def:eta} and $c \in \R$ is a uniquely determined constant. A similar result for degenerate diffusivities $\beta$ was obtained in \cite{BR22-invar-pr}. That $t\mapsto E(\mu_t)$ is decreasing is, at least heuristically, also implied by Remark \ref{rem:decreasing-along-E}. This suggests that our identification of the energy $E$ provides an ansatz to find Lyapunov functions --- and hence stationary solutions --- in more general cases than those covered in \cite{BR-IndianaJ}.
\end{rem}
\begin{rem}
	We would like to point out that if $\beta(r) = \sigma r $, $\sigma \in (0,\infty)$ and $b(r) = b_0 \in (0,\infty)$, then $\eta(r) = \frac{\sigma}{b_0}r(\log r - 1)$, $r \geq 0$, i.e. in this case $E$ in \eqref{def:E} is the classical Boltzmann entropy function.
\end{rem}

To conclude this subsection, we want to prove that in the situation of Theorem \ref{thm1} for each fixed $\nu(dx) = v(x)dx \in D'_{\text{nice}}$, the gradient field $\nabla^\Pscr_bE_\nu$ can be approximated by weighted smooth gradient fields in $L^2(\R^d,\R^d;\nu)$. To this end, we define
\begin{equation}
	G^\nu := \text{closure of } \big\{b(v)\nabla \zeta\,\big|\, \zeta \in C^\infty_c(\R^d)\big\}\text{ in }(L^2(\R^d,\R^d;\nu),\langle \cdot, \cdot \rangle_{b,\nu})
\end{equation}
(see \eqref{def:weighted-tensor2}). Then $T_\nu\Pscr = L^2(\R^d,\R^d;\nu) = G^\nu \oplus (G^\nu)^{\perp,\langle \cdot, \cdot\rangle_{b,\nu}}$, where $(G^\nu)^{\perp,\langle \cdot, \cdot\rangle_{b,\nu}}$ denotes the orthogonal complement of $G^\nu$ with respect to $\langle \cdot, \cdot \rangle_{b,\nu}$.
\begin{prop}\label{prop:Michael-grad}
	Let $E$ be as in \eqref{def:E} and $\nu(dx) = v(x)dx \in D_0$. Then $\nabla^\Pscr_b E_\nu \in G^\nu$.
\end{prop}
\begin{proof}
	We know by Theorem \ref{thm1} that $\nabla^\Pscr_b E_\nu = b(v)\nabla\big(g(v)+\Phi\big)$, with $\nabla\big(g(v)+\Phi\big) \in L^2(\R^d,\R^d;\nu)$. By Hypothesis 1 (i) we may assume $b \equiv 1$. Set $w:= g(v)+\Phi$ and $w_N:= (w\wedge N)\vee(-N)$, $N\in \N$. Then, as $N\to \infty$,
	\begin{equation*}
		\nabla w_N = \mathds{1}_{\{-N \leq w \leq N\}}\nabla w_N \xrightarrow{N \to \infty}\nabla w\text{ in }L^2(\R^d,\R^d;\nu).
	\end{equation*}
	Hence, since $\nu \in D'_{\text{nice}}$ (so, in particular, $\sqrt{v} \in H^1(\R^d)$), by \cite[Thm.3.1]{RZ94}, there exist $\zeta_n \in C^\infty_c(\R^d)$, $n \in \N$, such that 
	\begin{equation*}
		\zeta_n \xrightarrow{n \to \infty}w_N \text{ in }L^2(\R^d;\nu),
	\end{equation*}
	\begin{equation*}
		\nabla \zeta_n \xrightarrow{n \to \infty} \nabla w_N \text{ in }L^2(\R^d,\R^d;\nu),
	\end{equation*}
	and the assertion follows.
\end{proof}
\begin{kor}\label{cor:Michael}
	Let $D=b=0$, and $\mu_0$ and $t\mapsto \mu_t$ be as in Theorem \ref{prop:main-appl-general-case}. Then, for every $T>0$, \begin{equation*}
		\bigg(\nabla^\Pscr E_{\mu_t}\bigg)_{t \in [0,T]} \in \int_{[0,T]}^{\oplus} G^{\mu_t}\,dt,
	\end{equation*}
where the integral is meant in the sense of \cite{D69} (see also \cite{T79}). Moreover, for every $T>0$, the set of functions in $t \in [0,T]$
\begin{equation*}
	\bigg\{\psi \nabla^\Pscr F_{\mu_t}\big|\, \psi \in C^1([0,T],\R), \psi(T)=0,F\in \Fscr C^2_b\bigg\}
\end{equation*}
is dense in $\int_{[0,T]}^\oplus G^{\mu_t} \,dt$.
\end{kor}
\begin{proof}
	This is immediate from Theorem \ref{prop:main-appl-general-case} (iii) and Proposition \ref{prop:Michael-grad}.
\end{proof}
\begin{rem}
	\begin{enumerate}
		\item [(i)] We do not know whether in the case of \eqref{intro:classPME} we have $\nabla^\Pscr E_\nu \in G^\nu$ for all $\nu(dx) = v(x)dx \in D_{\text{nice}}$. The reason is that in this case we do not have $\sqrt{v} \in H^1(\R^d)$, so  \cite[Thm.3.1]{RZ94} does not apply.
		\item [(ii)] Corollary \ref{cor:Michael} is the first step to prove that for $0\leq s \leq t$
		\begin{equation}
			E(\mu_t) - E(\mu_s) = \int_s^t |\nabla^\Pscr E_{\mu_r}|^2_{\mu_r}\,dr,
		\end{equation}
	as in the finite-dimensional case and as purely (!) heuristic computations suggest, though, of course, by \eqref{gradflow-P} we have
	\begin{equation*}
		F(\mu_t)-F(\mu_s) = \int_s^t \langle \nabla^\Pscr_b E_{\mu_r},\nabla^\Pscr_b F_{\mu_r}\rangle_{b,\mu_r}\,dr, \quad \forall F\in \Fscr C^2_b.
	\end{equation*}
But, so far, we cannot show that this equality goes over to $E$ replacing $F$.
		\end{enumerate}
\end{rem}
\subsection{More general divergence-type equations with $x$-dependent drift}\label{sect:matrix-case}
Theorem \ref{thm1} can be extended to the more general divergence-type equation
\begin{equation}\label{eq:matrix-div-type-FPKE}
	\partial_t u = \divv \big(A(u)\nabla u\big)  -\divv \big( B(u)D(x)u\big),
\end{equation}
where $A, B: \R \to \R^{d\times d}$ are matrix-valued. More precisely, consider the following assumptions.
\\

\noindent \textbf{Hypothesis 2.}
\begin{itemize}
	\item[(i)] $B \in L^\infty(\R,\R^{d\times d})$, $|B| \geq b_0>0$, $B(r)$ is invertible for all $r \in \R_+$, and $B(r)^{-1} A(r) = \Psi(r)\Id$, where $\Psi:\R \to \R$, $\Psi \in C(\R)$ and $c_0 \leq \Psi \leq c_1$ for $c_i >0, i \in \{1,2\}$. 
	\item[(ii)] $D$ satisfies (iii) of Hypothesis 1.
\end{itemize} 
The proof of the following result is analogous to the proof of Theorem \ref{thm1}. Here, we define
$$D''_{\text{nice}} := \bigg\{\nu(dx) = v(x)dx \in \Pscr \,\big|\, v \in L^\infty(\R^d)\cap W^{1,1}_{\text{loc}}(\R^d), \frac{A(v)\nabla v}{ v} \in L^2(\R^d,\R^d;\nu)\bigg\}.$$
\begin{prop}\label{prop:matrix-case}
	Let Hypothesis 2 be fulfilled. Then any solution $t \mapsto \mu_t(dx) = u(t,x)dx$ to \eqref{eq:matrix-div-type-FPKE} such that $u(t)\in L^\infty(\R^d)\cap W^{1,1}_{\text{loc}}(\R^d)$, $\frac{A(u(t,\cdot))}{u(t,\cdot)}\nabla u(t,\cdot) \in L^2(\R^d,\R^d;\mu_t)$ and $\mu_t \in D(E)$ $dt$-a.s. is a solution to the gradient flow on $\Pscr$ with $E$,
	\begin{equation*}
		E: D(E) \subseteq \Pscr_a \to \R, \quad E(udx) := \int_{\R^d} \eta (u(x)) dx - \int_{\R^d}\Phi(x) u(x)dx,
	\end{equation*}
$$D(E):= \{\nu \in \Pscr_a\,|\, \nu(dx) = v(x)dx, \Phi \in L^1(\R^d;\nu), v \log v \in L^1(\R^d)\},$$
where \begin{equation}
	\eta(r):= \int_0^r \int_1^s \frac{\Psi(w)}{w}dw \,ds, \quad r \in \R,
\end{equation}
and with metric tensor 
\begin{equation}
\langle \cdot, \cdot \rangle_B : \nu \mapsto \langle\cdot,\cdot \rangle_{B,\nu} := \langle B(v)^{-1}\cdot, \cdot \rangle_{\nu}.
\end{equation}
The weighted gradient for $\nu(dx) = v(x)dx$ from $D''_{\text{nice}}\cap D(E)$
is
\begin{equation}
	\nabla^\Pscr_B E_\nu = \frac{A(v)\nabla v}{v}-B(v)D
\end{equation}
in $\big(L^2(\R^d,\R^d;\nu),\langle \cdot, \cdot \rangle_{B,\nu}\big)$.
\end{prop}
Existence of solutions to \eqref{eq:matrix-div-type-FPKE} with the properties from Proposition \ref{prop:matrix-case} seems an open question. In the simpler case that $D$ is independent of $x$, well-posedness of entropic solutions $u$ to \eqref{eq:matrix-div-type-FPKE} with $u(t) \in L^\infty(\R^d)\cap W^{1,1}_{\textup{loc}}$ has, e.g., been obtained in \cite{CP03}, but we do not know of a general result ensuring the additional assumptions of Proposition\ref{prop:matrix-case}.
	
\section*{Appendix \quad Differential geometry on $\Pscr$ via natural charts}\label{appendix}

In order to further convince the reader that the differential geometry from Section \ref{subsect:diff-geom} is natural, here we present its deduction from rigorous differential geometric principles. More precisely, starting with a natural chart on $\Pscr$, we obtain our geometry in analogy to classical Riemannian case. Since we consider $\Pscr$ as an infinite-dimensional manifold with the topology of weak convergence of probability measures, i.e. the initial topology of $\nu \mapsto \nu(\zeta)$, $ \zeta \in C^\infty_c(\R^d)$, a natural global chart  $\pi: \Pscr \to \R^\infty$ is
\begin{equation*}\label{eq:chart-pi}
	\pi: \Pscr \to \pi(\Pscr) \subseteq \R^\infty, \pi(\nu):= (\nu(g_i))_{i\in \N},
\end{equation*} 
where $\{g_i, i\in \N\} \subseteq C^\infty_c(\R^d)$ is dense.
The coordinate maps are $\pi^{(i)}(\nu):= \nu(g_i)$, $i \in \N$, and we also set 
$\pi_k(\nu):= (\nu(g_1),\dots,\nu(g_k)).$ There is no natural choice for $\{g_i, i \in \N\}$, i.e. the coordinates $\pi^{(i)}$ do not have intrinsic geometric meaning, and ultimately we will not rely on such a choice.
\paragraph{Test functions.}
The natural test functions on $(\Pscr,\pi)$ are cylindrical $C^1_b$-maps $f\circ \pi_k$, $k \in \N$, $f \in C^1_b(\R^k)$, i.e. we are led to consider the class
\begin{equation*}
	\tilde{\Fscr} C^1_b:= \big\{F: \Pscr \to \R\,|\,F(\nu) = f\big(\nu(g_1),\dots,\nu(g_k)\big), k \geq 1, g_i \text{ as above }, f \in C_b^1(\R^k)\big\}.
\end{equation*}
However, since the choice of $\{g_i\}_{i \geq 1}$ was arbitrary, it is reasonable to replace $g_i$ within the definition of $\tilde{\Fscr}C^1_b$ by arbitrary $h_i \in C^\infty_c(\R^d)$. Since our differential structure is of at most second order, we even allow $h_i \in C^2_c(\R^d)$, which leads to the test function class $\Fscr C^2_b$, see \eqref{test-fcts}.
\paragraph{Tangent spaces.}

In conjunction with the test function class $\tilde{\Fscr}C^1_b$, a curve $(-\varepsilon,\varepsilon)\ni t\mapsto \mu_t$ on $(\Pscr,\pi)$ is differentiable, if each $t\mapsto \pi^{(i)} \circ \mu_t$, $i \in \N$, is differentiable. For a differentiable curve with $\mu_0 = \nu$ and for $F \in \tilde{\Fscr}C^1_b$, set 
\begin{equation*}
	\frac{d}{dt}{\mu_t}_{|t =0}(F) := \ddt F(\mu_t)_{|t=0} = (\nabla f)(\nu(g_1),\dots,\nu(g_k))\cdot \ddt (\pi_k \circ \mu_t)_{|t=0},
\end{equation*}
and note
\begin{equation*}
	\ddt {\mu_t}_{|t = 0}(FG) = F(\nu)\ddt {\mu_t}_{|t = 0}(G)+G(\nu)\ddt {\mu_t}_{|t = 0}(F),\quad \forall F,G \in \tilde{\Fscr} C^1_b,
\end{equation*}
i.e. $\ddt {\mu_t}_{|t = 0}$ is a derivation. The equivalence class of differentiable curves $\tilde{\mu}$ with $\ddt( \pi_k \circ \tilde{\mu}_t) _{|t=0}= \ddt( \pi_k \circ \mu_t)_{|t=0}$ for all $k \in \N$ is denoted by $[\ddt \mu_t]$, and the tangent space $\tilde{T}_\nu\Pscr$ at $\nu$ is the set of all such equivalence classes. As usual, the associated tangent bundle is  $\tilde{T}\Pscr := \bigsqcup_{\mu \in \Pscr}\tilde{T}_\mu\Pscr$. In general, it is hard to further characterize elements of $\tilde{T}_\nu\Pscr$, and it is not clear whether these spaces are Hilbert.

Therefore, in order to obtain Hilbert tangent spaces, we restrict to suitable sub-tangent spaces $T_\nu\Pscr \subseteq \tilde{T}_{\nu}\Pscr$ as follows. For $\nu \in \Pscr$ and $\varphi \in L^2(\R^d,\R^d;\nu)$, the curve 
\begin{equation*}\label{curves-specific}
	(-1,1)\ni t\mapsto \mu^{\varphi,\nu}_t := \nu \circ(\Id+t\varphi)^{-1}
\end{equation*}
is differentiable with $\mu^{\varphi, \nu}_0 = \nu$, and 
\begin{equation}\label{calc-deriv-specific-curves}\tag{A.1}
	\ddt( \pi^{(i)}\circ \mu^{\varphi,\nu}_t) = \ddt\int_{\R^d}g_i(\Id+t\varphi)\,d\nu  = \int_{\R^d} \nabla g_i(\Id+t\varphi)\cdot \varphi \,d\nu.
\end{equation}
We set 
\begin{equation*}
	T_\nu\Pscr := \bigg\{\ddt\mu_t^{\varphi,\nu}, \varphi \in L^2(\R^d,\R^d;\nu)\bigg\},
\end{equation*}
and thus $T_\nu \Pscr \subseteq \tilde{T}_\nu \Pscr$ in the sense that $\big[\ddt\mu_t^{\varphi,\nu}\big] \in \tilde{T}_\mu\Pscr$ for all $\varphi \in L^2(\R^d,\R^d;\nu)$. $T_\nu\Pscr$ is isomorphic to $L^2(\R^d,\R^d;\nu)$, thus we identify $T_\nu \Pscr$ and $L^2(\R^d,\R^d;\nu)$ and endow $T_\nu \Pscr$ with the standard $L^2(\R^d,\R^d;\nu)$-inner product $\langle \cdot, \cdot \rangle_\nu$. Since \eqref{calc-deriv-specific-curves} still holds when one replaces $\pi^{(i)} \circ \mu^{\varphi,\nu}_t$ by $\mu^{\varphi, \nu}_t(h)$ for any $h \in C^2_c(\R^d)$, elements in $T_\nu\Pscr$ act as derivations not only on $\tilde{\Fscr}C^1$, but on $\Fscr C^2_b$. Hence we obtain the tangent spaces $T_\nu \Pscr$, tangent bundle $T\Pscr$ and metric tensor $\langle \cdot, \cdot \rangle: \nu \mapsto \langle \cdot, \cdot \rangle_\nu$ from \eqref{tangent-spaces-P}-\eqref{metric-tensor-P}.

\paragraph{Differential and Gradient.}
Following Riemannian differential geometry, for $\nu \in \Pscr$ and $F \in \tilde{\Fscr} C^1$, the differential of $F$ at $\nu$, $\tilde{\text{diff}}F_\nu: \tilde{T}_\nu \Pscr\to \R$, is a cotangent element with action on $[\ddt \mu_t] \in \tilde{T}_\nu\Pscr$ via
\begin{equation*}\label{diff-1}
	\tilde{\text{diff}}F_\nu\bigg(\big[\ddt \mu_t\big]\bigg) := \ddt \mu_t (F) = \ddt F(\mu_t)_{|t=0}.
\end{equation*}

For $\varphi \in T_\mu\Pscr \subseteq \tilde{T}_\mu\Pscr$ and $F(\nu) = f(\nu(g_1),\dots,\nu(g_k))$, we have
\begin{align*}\label{calc-diffF}
	\tilde{\text{diff}}F_\nu(\varphi) = \ddt F(\mu^{\varphi,\nu}_t)_{|t=0} = \sum_{i=1}^k\partial_i f(\nu(g_1),\dots,\nu(g_k))\langle \nabla g_i, \varphi \rangle_{\nu},
\end{align*}
and the same formula holds for $F \in \Fscr C^2_b$, i.e. when the coordinates $g_1, \dots, g_k$ are replaced by arbitrary $h_1,\dots,h_k \in C^2_c(\R^d)$. Hence we define $\text{diff}F_\nu$ for every $(\nu,F) \in \Pscr \times \Fscr C^2_b$ as in \eqref{diff-P}. Since our sub-tangent spaces $T_\mu \Pscr$ are Hilbert, there is a natural notion of gradient: For $F \in \Fscr C^2_b$, its gradient $\nabla^\Pscr F: \nu \mapsto \nabla^\Pscr F_\nu \in T_\nu\Pscr$ is the unique section in the tangent bundle $T\Pscr$ such that
\begin{equation*}
	\big\langle \nabla^\Pscr F_\nu,\varphi \big\rangle_{\nu} = \text{diff}F_\nu(\varphi),\quad \forall\varphi \in L^2(\R^d,\R^d;\nu),
\end{equation*}
which is the gradient from \eqref{grad-P}-\eqref{id-grad-diff-P}.
This characterization of the gradient is, of course, based on the Hilbert space structure of $T_\nu\Pscr$, which is one main reason for restricting to these sub-tangent spaces.

\paragraph{Competing Interests.}
The authors have no relevant financial or non-financial interests to disclose.

	\paragraph{Acknowledgement.} Funded by the Deutsche Forschungsgemeinschaft (DFG, German Research
	Foundation) – Project-ID 317210226 – SFB 1283. We also would like to thank Benjamin Gess for a very useful discussion on this paper.

\end{document}